\documentclass[12pt,a4paper]{article}%
\usepackage{amsmath}
\usepackage{graphicx}
\usepackage{epsfig}%
\usepackage{amsfonts}%
\usepackage{amssymb}
\usepackage{algpseudocode}						% algorithm envoronment, pseudocode
\usepackage{algorithm}	
\usepackage{color}
\usepackage{mathtools}
\usepackage{diagbox}						% table cells with diagonal subdivision
\newcommand{\cn}{\color{black}}

\newtheorem{theorem}{Theorem}[section]

\newtheorem{lemma}[theorem]{Lemma}

\newenvironment{proof}[1][Proof]{\noindent \emph{#1.} }{\hfill \
\rule{0.5em}{0.5em}}

\makeatletter\@addtoreset{equation}{section}\makeatother
\makeatletter\@addtoreset{figure}{section}\makeatother
\makeatletter\@addtoreset{table}{section}\makeatother
\begin{document}
 
\title{Tensor-based techniques for fast discretization and solution   
of 3D elliptic equations with \\ random coefficients}
 
\author{Venera Khoromskaia \thanks{Max-Planck-Institute for
        Mathematics in the Sciences, Inselstr.~22-26, D-04103 Leipzig,
        Germany ({\tt vekh@mis.mpg.de}).}
        \and
        Boris N. Khoromskij \thanks{Max-Planck-Institute for
        Mathematics in the Sciences, Inselstr.~22-26, D-04103 Leipzig,
        Germany ({\tt bokh@mis.mpg.de}).}
%         \and
%         Felix Otto \thanks{Max-Planck-Institute for
%         Mathematics in the Sciences, Inselstr.~22-26, D-04103 Leipzig,
%         Germany ({\tt otto@mis.mpg.de}).}
        }

\date{}

\maketitle

\begin{abstract}
In this paper, we propose and analyze the numerical algorithms 
for fast solution of periodic elliptic problems in random media in $\mathbb{R}^d$, $d=2,3$. 
We consider the stochastic realizations using checkerboard configuration of the equation coefficients 
built on a large $L \times L \times L$ lattice, where $L$ is the size of representative volume 
elements. The Kronecker tensor product scheme is introduced
for fast generation  of the stiffness matrix for FDM discretization on a tensor grid.
We describe tensor techniques for the construction of the  low Kronecker rank spectrally 
equivalent preconditioner in periodic setting to be used in the framework of PCG iteration. 
In our construction the diagonal matrix of the discrete Laplacian inverse represented 
in the Fourier basis is reshaped into a 3D tensor,
 which is then approximated by a low-rank canonical tensor, 
calculated by the multigrid Tucker-to-canonical tensor transform.
The FDM discretization scheme on a tensor grid is described in detail, and 
the computational characteristics in terms of $L$ for the 3D Matlab implementation 
of the PCG iteration are illustrated.
 The present work continues the developments in  \cite{KKO:17}, 
 where the numerical primer to study the asymptotic convergence rate vs. $L$ for the homogenized 
 matrix for 2D elliptic PDEs with random coefficients was investigated numerically.
  The presented elliptic problem solver 
 can be applied for calculation of long sequences of stochastic realizations in
 numerical analysis of 3D  stochastic homogenization problems for ergodic processes, 
 for  solving 3D quasi-periodic geometric homogenization problems, as well as 
 in the numerical simulation of dynamical many body interaction processes and multi-particle electrostatics.
  \end{abstract}



\section{Introduction}\label{Int:SH} 

Stochastic homogenization methods  provide means for calculating the average characteristics of the
  structural and geometric properties of random composites.  
 The  numerical schemes for solving elliptic partial differential equations (PDEs)
 with random input in the form of 
 stochastic/parametric elliptic equations have been intensively discussed in the literature
 \cite{FrSchTo:05,BrisLegoll:11,BoSchw:11,KuoSchSl:12,CaEhrLeSt,CaEhrLeStXi:18_2,CaEhrLeStXi:18,Fischer:18,KKO:17,BrisLeg:17}. 
 The theoretical analysis of quisi-periodic and stochastic/parametric problems can be found in 
 \cite{BeLiPa:78,Kozlov:79,JiKoOl:95,EngSoug:08,CoDeSch:11} and in references therein.
 The rank structured tensor methods for quasi-periodic geometric homogenization methods and for the 
 elliptic equations with highly oscillating coefficients were considered in 
 \cite{BokhSRep2:16,KaOsRaSch:20}.
 Data sparse and tensor methods for stochastic/parametric elliptic problems have been considered in 
 \cite{KhLitMat1:08,BoSchw:11,KhOsel_2_SPDE:10,SchGit:11,DoKhLiMa:15,DoSch:19}.

 The main computational challenge in stochastic homogenization techniques
 is that the exhausting of the important information 
 from the stochastic PDE (say,  homogenized coefficient matrix or solution, 
 and other important quantities of the stochastic process) 
 requires a huge number of realizations, 
 i.e. solving the target PDE many times for different  stochastic input. 
 In this respect, the valuable 3D stochastic simulations presuppose the 
 strong requirements to the numerical efficiency of the chosen 3D elliptic problem solver.
 
 This paper continues the development of efficient algorithms 
 initiated by a numerical primer in \cite{KKO:17}  for fast solution
  of the 2D elliptic PDEs in random media, where the computational scheme 
  for stochastic realizations using   the  general overlapping-type coefficient profile has been 
 developed. The numerical study in \cite{KKO:17} confirmed the theoretical convergence rate 
 for the homogenized coefficients matrix in the size of representative volume element (RVE), presented in
 \cite{GloOt:11,GlOtto:12,GlOtto:15,GlOtto:16}.
 %In what follows, we introduce the fast solver for the class of 3D elliptic PDEs in random media.
   Recall that the algorithms described in \cite{KKO:17} are capable for 2D calculations with 
  the number of realizations limited by the order of $M=10^5$, implemented for 
  coefficient configuration built on $L\times L$ lattices with the RVE size $L$ up to $128$. 
  However, the 3D calculations by using 
 a general overlapping-type profile for generation of random coefficients seem to be prohibitive
 for the large number of coefficient realizations over $L \times L \times L$ lattice structures. 
%   in the required range  of the RVE size, $L$, and the univariate 
%   grid parameter of the order $n\approx 10 L$.  
%      
  
 In this paper, we describe the numerical scheme  for discretization and 
 solution of the $d$-dimensional 
 stochastic homogenization problems for $d=2,3$, which  employs the 
 realizations over a checkerboard type  configuration of the stochastic coefficient
 on the $L \times L$ or  $L \times L \times L$ lattice, respectively. 
  In 3D case, we use   the product  piecewise linear finite elements on the  $n \times n \times n$ 
  Cartesian grids with $n=n_0 L$, $n_0=4,8,16,\dots$, assuming that the 
  jumps in the equation coefficient are resolved by  
  non-overlapping square subdomains (unit cells). 
 We introduce a tensor-based scheme for fast generation of the stiffness 
 matrix for both 2D and 3D problems by using the Kronecker product construction for assembling  
 of the FEM stiffness matrix and for the design of the rank-structured  preconditioner. 
 In the 3D case, we construct the spectrally equivalent preconditioner by
 employing   the explicit representation (approximation) 
 of the 3D periodic Laplacian operator inverse 
 in the Fourier basis  in a form of a short sum  of the three-fold 
 Kronecker products of $n \times n$ matrices, similar to \cite{HeidKh2Sch:18} where the Laplacian 
 with Dirichlet boundary conditions was considered. Algorithmically, in our construction  
 the diagonal matrix of the discrete Laplacian inverse represented in the Fourier basis is reshaped into a 3D tensor,
 which is then approximated by a low-rank canonical tensor, see Lemma \ref{lem:RankPinverse}. 
 This approximation is calculated by
 using the multigrid Tucker-to-canonical tensor transform introduced in \cite{KhKh3:08,Khor_bookQC_2018}.

 The presented numerical scheme with the  checkerboard type coefficients in 2D leads 
 to a much faster method as compared with  that for overlapping coefficients \cite{KKO:17}. 
 This allows us to perform computations with RVE size $L$ up to $L=512$ for 2D problems
 and for the number of realizations of the order of $2^{15}$.
 For 3D case the large number realizations,  $M$, for RVE size up to  $L=32, 64$, discretized 
 on $n\times n \times n$ grids $n=n_0 L$ can be calculated. 
  The proposed tensor-based numerical  techniques enable computations of the   descriptive
series of stochastic realizations  for $2D$ and $3D$ problems in a wide range 
of the RVE size $L$,  % for  the respective number of realizations $N$, 
using MATLAB on a moderate computer cluster. 
% In 3D case the computations of long sequences of 
% realizations for large $L$ require the extended computational facilities,
% say, multiprocessor systems.

The proposed elliptic problem solver 
 can be applied for the numerical analysis of 3D  stochastic homogenization problems for ergodic processes
 with variable  contrast in random coefficients, 
 for solving numerically stochastic elliptic PDEs in random heterogeneous materials in $\mathbb{R}^3$, 
 for fast solution of quasi-periodic (multi-scale) geometric homogenization problems for elliptic equations, in  
 the computer simulation of dynamical many body interaction processes and multi-particle electrostatics,
 as well as for numerical analysis of optimal control problems in random media.

  The rest of the paper is organized as follows. In Section \ref{Int2:SH}, we describe the 
  problem setting and specify the particular schemes for random generation of stochastic coefficients.  
  Section \ref{sec:Comp_Ahom} presents the main computational approach, where
  \S\ref{ssec:Discret_Scheme} describes the discretization scheme, 
  \S\ref{ssec:FDM_Kron} outlines the matrix generation by using Kronecker product sums
  and \S\ref{ssec:FDM_Kron_stoch} sketches the method for fast matrix assembling of the 
  stochastic part. Section \ref{sec:PCG} describes  the construction of the efficient
  low Kronecker-rank spectrally close preconditioner in the PCG iteration for solving   
  elliptic problems with variable coefficients arising for stochastic realizations with fixed 
  value of RVE size, $L$,  and the univariate grid size $n$.
  Section \ref{sec:Numer_2D_3D} presents the results of numerical experiments 
  demonstrating the asymptotic complexity and timing of the Matlab implementation.
  In the spirit of \cite{KKO:17}, we also verify numerically the standard estimates 
  on the asymptotic convergence rate
  for the simple average (standard deviation)  of the homogenized coefficient matrix
  for 2D and 3D stochastic simulations with the checkerboard-type realization of coefficients.

\section{General problem setting }\label{Int2:SH} 
 
For given $f\in L^2(\Omega)$ such that $\int_{\Omega} f(x) dx =0$,
we consider the model elliptic boundary value problems on $\Omega :=[0,1)^d$, for $d=2,3$,
\begin{equation} \label{eqn:hm_setting}
  {\cal A}\varphi := -\nabla \cdot \mathbb{A}(x)\nabla \varphi =f(x), \quad 
x=(x_1,\ldots,x_d)\in \Omega,
\end{equation}
endorsed with periodic boundary conditions on $\Gamma = \partial \Omega $. 
The diagonal $d\times d$ coefficient matrix $\mathbb{A}(x)$ is defined by
\[
 \mathbb{A}(x)= a(x) I_{d\times d}, \quad x\in\Omega,
\]
where the scalar piecewise constant function $a(x)>0$ is generated randomly for every 
stochastic realization defined by the size $L$ of RVE, such that it has many jumps in $\Omega$,
see Figure \ref{fig:2D_example_Check} for the example in 2D case. 
There are many computational approaches for solving 
elliptic PDEs with random input, see \cite{BrisLegoll:11,CaEhrLeStXi:18,BrisLeg:17,KKO:17} 
and references therein. In particular, 
in \cite{KKO:17} the fast elliptic problem solver in 2D case was applied to study  numerically
 the convergence properties of the stochastic homogenization techniques,
providing means to substitute the stochastic coefficient $\mathbb{A}(x)$ by its simple 
homogenized version $\bar{\mathbb{A}}_L\in \mathbb{R}^{2\times 2}$,
such that for large values of RVE size $L$ the average 
quantities over  the long sequence of stochastic realizations will be very close to its 
homogenized version.
\begin{figure}[htb]
\centering
\includegraphics[width=7.0cm]{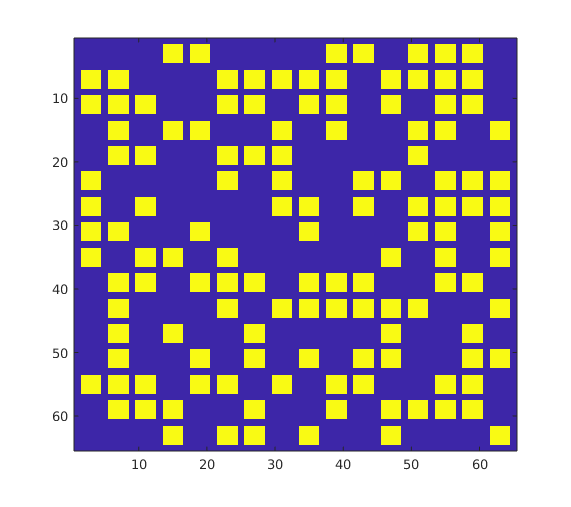}\quad
\includegraphics[width=7.2cm]{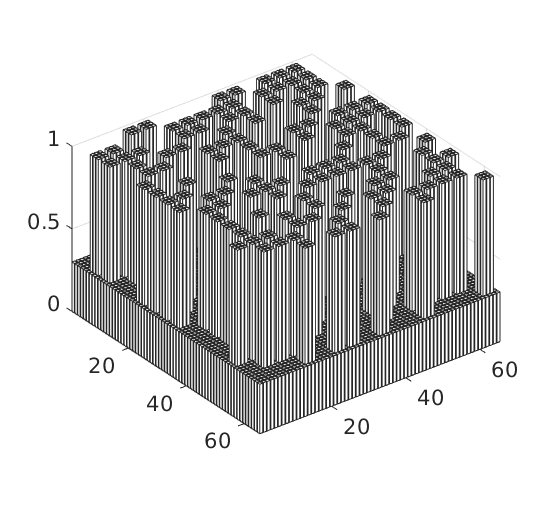} 
 \caption{\small Example of stochastic realization of coefficient for $L\times L$ lattice
 with $L=16$ with $\lambda=0.2$ and fixed contrast  parameter $\beta=0.8$.}   
\label{fig:2D_example_Check}
\end{figure}

In this paper, we describe the new discretization and solution scheme for 
solving the $d$-dimensional problems
with checkerboard type of random coefficients configuration $\mathbb{A}(x)$ for $d=2,3$.
We consider the sequence of $M$ stochastic realizations specifying the variable part 
in the $d\times d$ coefficient matrix
 $\widehat{\mathbb{A}}_m(x)$, $m=1,\ldots,M$. For ease of exposition, we discuss the 
 3D problems, $d=3$, and first, consider the case of constant scaling parameter $\lambda$.  
 Fixed coefficient $\widehat{\mathbb{A}}_m(x)$ and the scaling parameter $0 < \lambda\leq 1$,
 we solve the periodic  elliptic problems in $\Omega =[0,1)^3$,
 \begin{equation}\label{eqn:PDE_stoch_stand}
  {\cal A}_m \phi := -\nabla \cdot \mathbb{A}_m(x)\nabla \phi =f(x), \quad 
x=(x_1,\ldots,x_d)\in \Omega,
 \end{equation}
where the matrix-valued equation coefficient is specified by
 \[
 \mathbb{A}_m(x) = \lambda I_{d\times d}+ \beta \widehat{\mathbb{A}}_m(x)= a_m(x) I_{d\times d},
 \]
 with $\beta= 1-\lambda$, and the diagonal entry in $\mathbb{A}_n(x)$ is defined by
 \begin{equation}\label{eqn:coef_diag}
{a}_m(x) = \lambda + \beta \widehat{a}_m(x).
\end{equation}
% that is 
%  \[
%   \widehat{\mathbb{A}}_m(x)=\mbox{diag}\{\widehat{a}_m(x),\widehat{a}_m(x),\widehat{a}_m(x)\}.
% \]

 Notice that in the application to numerical estimation of the homogenized matrix, 
 see \cite{KKO:17}, the triple of 
 elliptic equation has to be solved for every stochastic realization. Specifically, 
 for $i=1,2,3$ the periodic  elliptic problems in $\Omega =[0,1)^3$,
 \begin{equation}\label{eqn:RHS_grad}
 -\lambda \Delta \Phi_i - \beta\nabla \cdot \widehat{\mathbb{A}}_m(\cdot) ({\bf e}_i + \nabla \phi_i)=0,
  %-\nabla \cdot A(\cdot) \nabla \Phi_i = \nabla \cdot A(\cdot) {\bf e}_i,
 \end{equation}
 where $\beta= 1-\lambda$, and the unit vectors ${\bf e}_i$, $i=1,2,3$, are given by 
 $${\bf e}_1=(1,0,0)^T, \quad {\bf e}_2=(0,1,0)^T, \quad {\bf e}_3=(0,0,1)^T.
 $$ 
 The right-hand side in (\ref{eqn:PDE_stoch_stand}), 
 rewritten in the canonical form (\ref{eqn:hm_setting}), is represented by 
 \[
  f_i(x)= \beta \nabla \cdot \widehat{\mathbb{A}}_m(x){\bf e}_i,
 \]  
where the diagonal coefficient 
is defined in terms of the scalar function $\widehat{a}_m(x)$,
 $\widehat{\mathbb{A}}_m(x)= \widehat{a}_m(x) I_{d\times d} $. 
Hence, we arrive at the representations for the right-hand sides
\begin{equation}\label{eqn:RHS_A_hom_m}
 f_1(x)= \beta\dfrac{\partial \widehat{a}_m(x)}{\partial x_1}, \quad
 f_2(x)=\beta \dfrac{\partial \widehat{a}_m(x)}{\partial x_2}, \quad
 f_3(x)=\beta \dfrac{\partial \widehat{a}_m(x)}{\partial x_3}.
\end{equation}
 
%  The equation (\ref{eqn:RHS_grad}) can be recast in the form of (\ref{eqn:hm_setting}) as follows
% \begin{equation}\label{eqn:PDE_stoch_stand}
%   {\cal A}_m \phi_i := -\nabla \cdot \mathbb{A}_m(x)\nabla \phi_i =f_i(x), \quad 
% x=(x_1,\ldots,x_d)\in \Omega,
%  \end{equation}
% where the matrix-valued equation coefficient is specified by
%  \[
%  \mathbb{A}_m(x) = \lambda I_{d\times d}+ \beta \widehat{\mathbb{A}}_m(x)= a_m(x) I_{d\times d},
%  \]
% and the diagonal entry in $\mathbb{A}_n(x)$ is defined by
%  \begin{equation}\label{eqn:coef_diag}
% {a}_m(x) = \lambda + \beta \widehat{a}_m(x).
% \end{equation}
% % that is 
% %  \[
% %   \widehat{\mathbb{A}}_m(x)=\mbox{diag}\{\widehat{a}_m(x),\widehat{a}_m(x),\widehat{a}_m(x)\}.
% % \]
%  The right-hand side in (\ref{eqn:PDE_stoch_stand}), 
%  rewritten in the canonical form (\ref{eqn:hm_setting}), is represented by 
%  \[
%   f_i(x)= \beta \nabla \cdot \widehat{\mathbb{A}}_m(x){\bf e}_i,
%  \]  
% where the diagonal coefficient 
% is defined in terms of the scalar function $\widehat{a}_m(x)$,
%  $\widehat{\mathbb{A}}_m(x)= \widehat{a}_m(x) I_{d\times d} $. 
% Hence, we arrive at the representations for the right-hand sides
% \begin{equation}\label{eqn:RHS_A_hom_m}
%  f_1(x)= \beta\dfrac{\partial \widehat{a}_m(x)}{\partial x_1}, \quad
%  f_2(x)=\beta \dfrac{\partial \widehat{a}_m(x)}{\partial x_2}, \quad
%  f_3(x)=\beta \dfrac{\partial \widehat{a}_m(x)}{\partial x_3}.
% \end{equation}

 Figure   \ref{fig:2D_example_Check} shows an example of stochastic realizations, which 
specify the locations of jumps in the equation coefficient $a(x)$ in 2D case for $L=16$.

In the previous simple model, we determine randomly the positions of jumps in the coefficients and
use the constant length for the stochastic inclusions $\beta=1-\lambda$, which we 
call by the {\it contrast constant}. 
Our scheme also applies to the case of varying contrast constant $\beta=\mu(x)$ that may
vary in the interval $\beta \in [0,\beta_0]$, $\beta_0=1-\lambda$, randomly for each realization.
Figures \ref{fig:2D_example_two_rand} and \ref{fig:2D_example_two_layer} illustrate the stochastic 
realizations of coefficients with two contrast parameters $\beta_1=0.3$ and $\beta_2=0.6$, 
which are randomly distributed or have layer type structure.
\begin{figure}[htb]
\centering
\includegraphics[width=7.0cm]{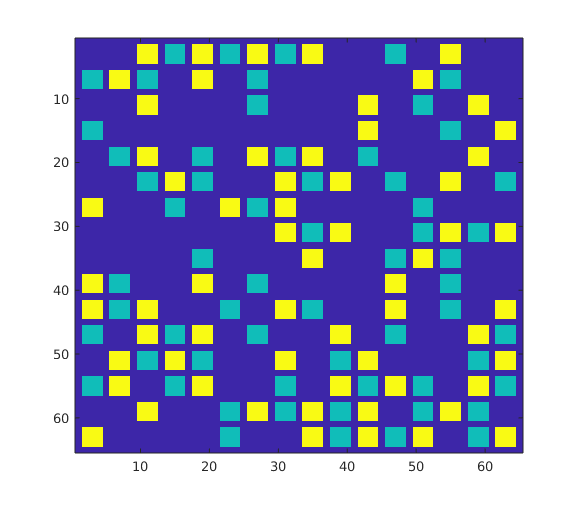}\quad
\includegraphics[width=7.2cm]{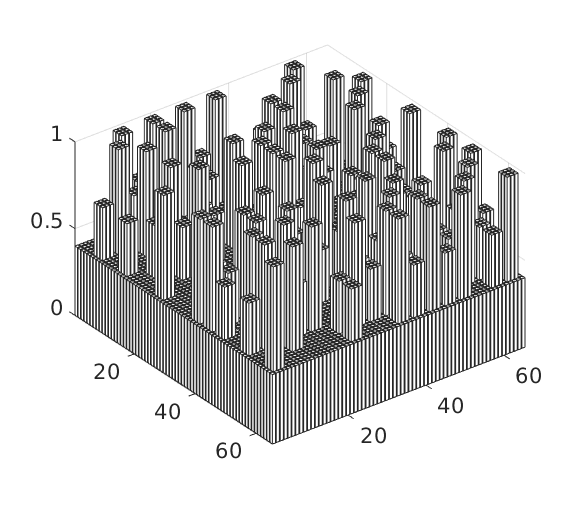} 
 \caption{\small Example of stochastic realization of coefficient for $L\times L$ lattice
 with $L=16$ with two contrast  parameters $\beta_1=0.3$ and $\beta_2 =0.6$.}   %$\alpha =1/4$, $\lambda=0.4$.}
\label{fig:2D_example_two_rand}
\end{figure}

\begin{figure}[htb]
\centering
\includegraphics[width=7.0cm]{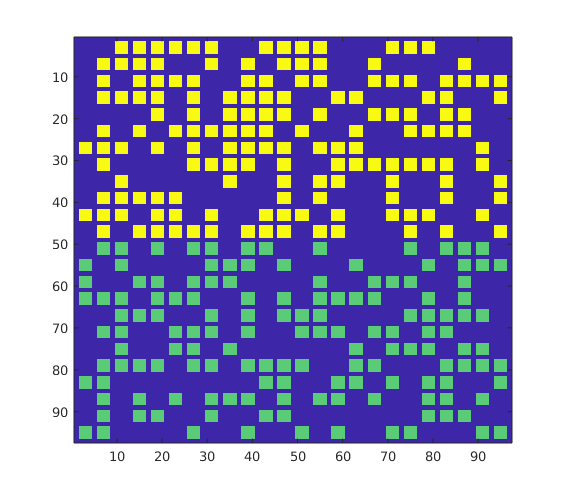}\quad
\includegraphics[width=7.2cm]{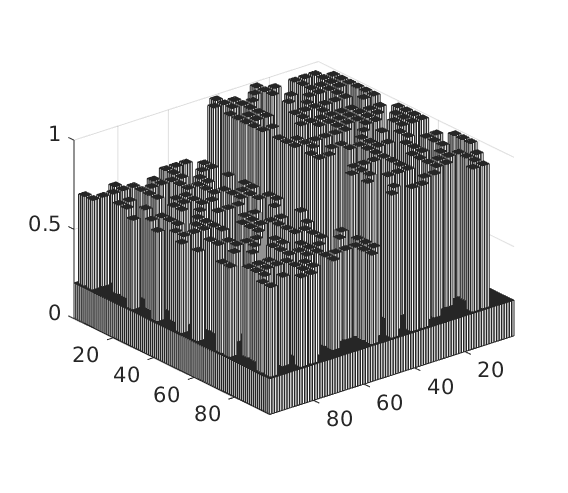} 
 \caption{\small Example of stochastic realization of coefficient for $L\times L$ lattice
 with $L=24$ with layer-type contrast  parameters $\beta=0.6$ and $\beta=0.3$.}  
\label{fig:2D_example_two_layer}
\end{figure}

 We are interested in the construction of fast numerical solution of the equation (\ref{eqn:RHS_grad})
with coefficients $\mathbb{A}(x)$, generated in the course of stochastic realizations.
  % in calculation of various functionals on the sequence of solutions.  
In this problem setting the bottleneck task is   fast generation of the (large) FEM stiffness 
matrix in a sparse format, see \cite{KKO:17}, which should be re-calculated many thousands times
for   long sequences \cn of stochastic realizations. Here the computational challenges are twofold:
\begin{itemize}
 \item[(A)] Fine $n \times n \times n$-grids required for the resolution of coefficients on large 
 $L\times L\times L$ lattice.
 \item[(B)] Large number of stochastic realizations $M$ of the order of $10^4-10^5$ that are necessary for 
 reliable estimation of desired stochastic quantities.
\end{itemize}
In item (A) the construction of new FEM discretization and then generation of large stiffness 
matrix in data sparse format is required for every stochastic realization. 
Item (B) requires the fast iterative solver for the arising linear systems of equations, 
which should be robust with respect to 
the main model parameters $L$, $N=n^d$, and the random equation coefficients.
Our techniques suggest the effective approach for solving both problems (A) and (B). 

 \begin{figure}[htb]
\centering
\includegraphics[width=7.8cm]{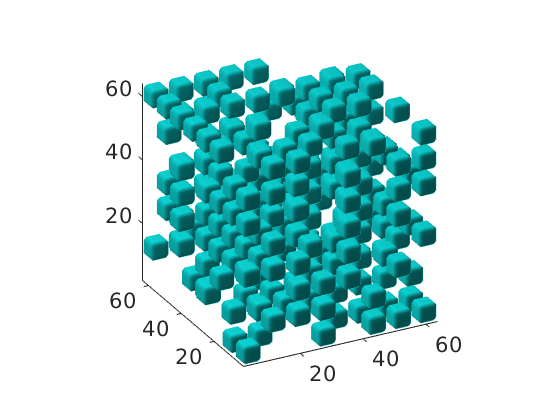}
\includegraphics[width=7.8cm]{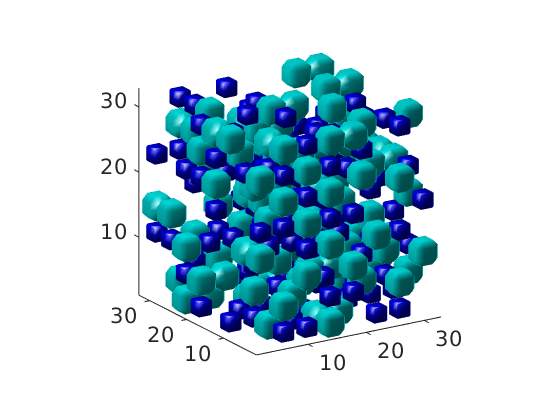}
 \caption{\small Example  of 3D stochastic coefficient on
   $L\times L\times L$ lattice  with $L=8$ with fixed contrast $\lambda$ (left), and two randomly 
   distributed contrast parameters (right).} 
   \label{fig:3D_coef_L8}
\end{figure}
Figure \ref{fig:3D_coef_L8} illustrates the configuration of the matrix coefficient visualized 
 for an example of 3D realizations on the  $L\times L \times L$ lattice, with $L=8$,  
   The number of inclusions is about $\frac{1}{2} L^3$.

In what follows, we describe both fast and memory-efficient discretization and solution method 
for the class of stochastic PDEs specified above, 
which allows the reliable numerical estimate of the mean (homogenized) constant coefficient in 
 the system   (\ref{eqn:RHS_grad}) for $d=2,3$     
 for rather large value of RVE $L$ and various model parameters at the limit of large $M \to \infty$, 
  see \cite{GlOtto:12,GlOtto:16,KKO:17}. This approach also allows to effectively estimate the average 
  solution of the 3D stochastic PDE (\ref{eqn:hm_setting}) with the given right-hand side $f(x)$ by using the 
  precomputed homogenized coefficient matrix.

 \section{Computational scheme for the stochastic average}
 \label{sec:Comp_Ahom}
 
 \subsection{Galerkin FEM discretization scheme }\label{ssec:Discret_Scheme} 
 
 The discretization scheme for the $d$-dimensional problem is constructed by FEM on tensor grid in 
 $\Omega \in \mathbb{R}^d$  similar to  the 2D case described in \cite{KKO:17}. We consider 
 the RVE approximation specified by the checkerboard-type realizations of
 the random coefficient field on the $L^{\otimes d}$ tensor product 
 lattice\footnote{That is $L\times L$ lattice for $d=2$, and $L\times L \times L$ lattice for $d=3$.}.
 This lattice is composed of $L^d$ unit cells $G_s$ such that
 \[
  \Omega = {\bigcup}_{s=1}^{L^d} G_s.
 \]
 In 2D case the FEM discretization and the construction of the Galerkin matrix 
 can be viewed as a special case of the more general scheme in \cite{KKO:17} based 
 on the ``overlapping'' type random realizations of the coefficient field.

 Given the number $n_0= 2^{p_0}$ with $p_0=2,3,4,...$, of the grid intervals specifying 
 the size a unit cell, 
  we introduce the uniform $n_1^{\otimes d}$ rectangular grid $\Omega_{h}$ in $\Omega=[0,1)^d$
 with the grid size $h=\frac{1}{n_1 -1}$, such that $n_1 = n_0 L +1$, i.e.,
 $h = \frac{1}{n_0 L}$. We assume that the ``unit cell'' $G_s$, $s=1,\ldots, L^d$, 
 includes the square ``unit sub-cell'' $S_s\subseteq G_s$ of size $(\frac{2\alpha}{L})^{\otimes d}$ 
 (that is  $\frac{2\alpha}{L} \times \frac{2\alpha}{L} \times \frac{2\alpha}{L}$ for $d=3$) 
 which adjusts the square grid $\Omega_{h}$, such that
 the center $c_s$ of $S_s$ is located at the center of $G_s$.
 The sub-cell $S_s$ denotes the region where the stochastic realization is allowed to 
 generate the jumping coefficient. The number of unit cells $K \leq L^d$ where the coefficient is perturbed 
 varies for different stochastic realizations.
 %belongs to the set of grid points in $\Omega_h$,
 The overlap factor $0< \alpha \leq \frac{2^{p_0 -1}}{n_0}$ may take values 
 $\alpha \in \{\frac{1}{n_0}, \frac{2}{n_0},\ldots \frac{2^{p_0 -1}}{n_0}\}$ depending on the choice of 
 $p_0$.   
 In this construction the univariate size of the unit sub-cell $S_s$ varies as
 \[
  \frac{2\alpha}{L} = \frac{2\alpha n_0}{n_0 L} = k h, \quad \mbox{with } \quad
  k=2, 4, \ldots, n_0.
 \]
 In the presented numerical examples we normally use the overlap constant $\alpha =1/4$ or 
 $\alpha =1/2$.
 For $\alpha =1/2$, the maximal size of the unit sub-cell is given by $({1}/{L})^{\otimes d} $, which  
contains $n_0+1$ grid points in each spacial direction leading to $n_1^{\otimes d} $ rectangular
grid with $n_1=n_0 L +1$.

Fixed $L$, the FEM discretization of the elliptic PDE in (\ref{eqn:PDE_stoch_stand}) can be constructed, 
in general, on a sequence of dyadic refined grids by choosing $p_0=2,3, \ldots $, so that the increase 
of the parameter $p_0$ improves the accuracy of FEM approximation.

Given a finite dimensional space $X\subset H^1(\Omega)$ of tensor product piecewise linear finite elements 
$X = \mbox{span}\{ \psi_\mu(x) \}$ associated with the grid $\Omega_h$, with  $\mu=1,...,N_d $, $N_d=n_1^d$,
incorporating periodic boundary conditions,  we are looking for the traditional 
FEM Galerkin approximation of the exact solution in the form 
$$
\phi(x) \approx \phi_X(x)=\sum_{\mu=1}^{N_d} u_\mu \psi_\mu(x) \in X,
$$
where ${\bf u}=(u_1,\ldots,u_{N_d})^T\in  \mathbb{R}^{N_d}$ denotes the unknown coefficients vector.
Fixed realization of the coefficient $a^{(m)}(x)$, %for $i=1,2, \ldots,d$,
we define the Galerkin-FEM discretization in $X$ of the variational 
formulation of equation (\ref{eqn:PDE_stoch_stand}) by
%(\ref{eqn:Ellipt_variat})  by %, see also (\ref{eqn:scaled_setting}) 
\begin{equation}\label{eqn:FEM_Galerk}
 A {\bf u} = {\bf f}, \quad A = [a_{\mu \nu}]\in \mathbb{R}^{N_d\times N_d},\quad 
 {\bf f}=[f_\mu] \in  \mathbb{R}^{N_d}, %  \quad N=n_1^2,
\end{equation}
where the Galerkin-FEM stiffness matrix $A$ generated by the equation coefficient 
 ${\mathbb{A}}^{(m)}(x)$ is calculated by using the associated bilinear form
\begin{equation}\label{eqn:FEM_discr}
a_{\mu \nu} = \langle {\cal A} \psi_\mu,  \psi_\nu \rangle = 
\int_{\Omega}(\lambda \nabla \psi_\mu  \cdot \nabla \psi_\nu + 
\beta a^{(m)}(x) \nabla \psi_\mu \cdot \nabla \psi_\nu)dx, 
%\quad f_i = \int_\Omega f\, \psi_i dx.
\end{equation}
and $f_\mu = \langle f, \psi_\mu \rangle$.

In specific application to numerical estimation of the homogenized coefficient matrix 
the corresponding right-hand side is defined 
\begin{equation}\label{eqn:FEM_discr_RHS}
 f_{\mu,i} = \langle f, \psi_\mu \rangle= \beta
 \int_\Omega \nabla \cdot \widehat{a}^{(m)}(x)\, {\bf e}_i \psi_\mu \, dx =
 - \beta \int_\Omega \widehat{a}^{(m)}(x) \frac{\partial \psi_\mu}{\partial x_i} dx, \quad i=1,2, \ldots,d.
\end{equation}

Corresponding to (\ref{eqn:coef_diag}) and (\ref{eqn:FEM_discr}), 
 we represent the stiffness matrix $A$ in the additive form
\begin{equation}\label{eqn:FEM_matrix}
A=\lambda A_\Delta + \beta \widehat{A}_s,
\end{equation}
where $A_\Delta$ represents the $N_d \times N_d$
FEM Laplacian matrix in periodic setting that has the standard $d$-term Kronecker product form. 
Here matrix $\widehat{A}_s$ provides the FEM approximation to the "stochastic part" 
in the elliptic operator corresponding to the coefficient $\widehat{a}^{(n)}(x)$, 
see (\ref{eqn:coef_diag}). The latter is determined by the sequence of
random coefficient distributions in the course of stochastic realizations, 
numbered by $m=1,\ldots,M$.

In the case of complicated jumping coefficients the stiffness matrix generation in the elliptic FEM 
usually constitutes the dominating part of the overall solution cost.
The asymptotic convergence of the stochastic homogenization process presupposes that
the equation (\ref{eqn:FEM_Galerk})
has to be solved many hundred or even thousand times, so that for every realization
one has to update the stiffness matrix $A$ and the right-hand side $\bf f$.

Our discretization scheme  computes  all matrix entries at a low cost
by assembling the local Kronecker products of sparse matrices obtained by representation of
$\widehat{a}^{(m)}(x)$ as a sum of separable functions.
This allows to store the resultant stiffness matrix in the sparse matrix format. 
Such a construction only includes the 
pre-computing of small tri-diagonal matrices representing 1D elliptic operators 
with jumping coefficients in periodic setting.
In the following sections, we shall describe the efficient construction 
of the "stochastic" term $\widehat{A}_s$ to be updated for every realization.

\subsection{Matrix generation by using Kronecker product sums}\label{ssec:FDM_Kron} 

To enhance the time consuming matrix assembling process 
we apply the FEM Galerkin discretization (\ref{eqn:FEM_discr}) of 
equation (\ref{eqn:PDE_stoch_stand}) by means of the 
tensor-product piecewise linear finite elements 
$$
\{\psi_{\boldsymbol{\mu}}(x):=\psi_{\mu_1}(x_1) \cdots \psi_{\mu_d}(x_d)\}, 
\quad {\boldsymbol{\mu}}=(\mu_1,\ldots,\mu_d), 
\quad \mu_\ell\in {\mathcal I}_\ell=\{1,\ldots,n_\ell\}, % \; \ell=1,\ldots,d,
$$
for $\ell=1,\ldots,d,$
where $\psi_{\mu_\ell}(x_\ell)$ are the univariate piecewise linear hat 
functions. Notice that the univariate grid size $n_\ell$ is of the order of $n_\ell=O(1/\epsilon)$,
where the small homogenization parameter is given by $\epsilon\approx 1/(n_0 L)$,
designating the total problem size 
\[
N_d = n_1 n_2\cdots n_d=O(1/\epsilon^d).
\]
The $N_d\times N_d$ stiffness matrix is constructed by the standard mapping of the multi-index 
$ \boldsymbol{\mu}$
into the long univariate index $1\leq \mu \leq N_d$ for the active degrees of freedom in periodic setting. 
For instance,  we use the so-called big-endian convention for $d=3$ and $d=2$
\[
 {\boldsymbol{\mu}}\mapsto \mu:= \mu_3 + (\mu_2-1)n_3 + (\mu_1-1)n_2 n_3, 
 \quad {\boldsymbol{\mu}}\mapsto \mu:= \mu_2 + (\mu_1-1)n_2,
\]
respectively. 
We first consider the case $d=2$ in more detail. 

We calculate the stiffness matrix  
by assembling of the local Kronecker product terms by using representation of the ``stochastic part''
in the coefficient $\widehat{a}^{(m)}(x)$ as an $R$-term sum of separable functions.
 This leads to the linear system of equations
\begin{equation} \label{eqn:FEM_syst}
 A {\bf u} = {\bf f},
\end{equation}
constructed for the general $R$-term separable coefficient $a(x_1,x_2)$ with $R\geq 1$.

By simple algebraic transformations (e.g. by lamping of the mass matrices)
the matrix ${A}$ can be represented in the form (without loss of approximation order)  %\cite{..}
\begin{equation} \label{eqn:Lapl_Kron_D}
 {A} \mapsto A = A_1 \otimes D_2 + D_1 \otimes A_2,
\end{equation}
 where $D_1, D_2$ are the diagonal matrices with positive entries, 
 and $\otimes$ means the Kronecker product of matrices,  see the discussion in \cite{KKO:17}.
 This representation applies, in particular, to the periodic Laplacian.
For example, in the case of anisotropic Laplacian 
the representation in (\ref{eqn:Lapl_Kron_D}) can be further simplified to 
\begin{equation*} \label{eqn:Lapl_Kron}
 {A} \mapsto B = \alpha_2 A_1 \otimes I_2 + \alpha_1 I_1 \otimes A_2,
\end{equation*}
which will be used as a prototype preconditioner for solving the target linear system (\ref{eqn:FEM_syst}).

Taking into account the rectangular structure of the grid, 
we use the simple finite-difference (FD) scheme for the matrix
representation of the Laplacian operator $\Delta$.
The scaled discrete Laplacian incorporating periodic boundary conditions takes the form
\begin{equation}\label{eqn:Lap_Kron}
 A_\Delta  = \Delta_{1,P} \otimes I_{n_2} + I_{n_1}\otimes \Delta_{2,P},
\end{equation}
where, say, in the variable $x_1$ we have
\[
 \Delta_{1,P} = -\mathrm{tridiag} \{ 1,-2,1 \} + P^{(1)} \in \mathbb{R}^{n_1\times n_1},
\] 
such that the entries of the "periodization" matrix $P^{(1)}\in \mathbb{R}^{n_1\times n_1}$ 
are all zeros except 
$$ 
P^{(1)}_{1,n_1}=P^{(1)}_{n_1,1}=1, \quad \mbox{and} \quad  P^{(1)}_{1,1}=P^{(1)}_{n_1,n_1}=-1,
$$ 
see (\ref{BC_Neu_peri}), right.
Here
$I_{n_1}\in \mathbb{R}^{n_1 \times n_1}$ is the identity matrix, $\Delta_{1,P}$ and $\Delta_{2,P}$
are the 1D finite difference Laplacians in variables $x_1$ and $x_2$, respectively  
(endorsed with the Neumann boundary conditions).
 We say that the Kronecker rank of the matrix $A$ in (\ref{eqn:Lap_Kron})
equals to $2$, $rank_{Kron}(A)=2$.

For the assembling of the stiffness we also need the 1D Laplacian with Neumann boundary conditions.
To that end we notice that the $n_1 \times n_1$ Laplacian matrices for the Neumann and periodic boundary 
conditions in the first 1D variable read as 
 \begin{equation}\label{BC_Neu_peri}
 \Delta_{1,N} = 
 \begin{bmatrix}
  -1 &  1 &  \cdots &  0 &  0\\
   1 & -2 & \cdots &  0 &  0 \\
   \vdots & \vdots & \ddots  &  \vdots & \vdots \\ 
   0 &  0 &  \cdots & -2 &  1 \\ 
   0 &  0 &  \cdots &  1 & -1 \\ 
 \end{bmatrix}
 \quad \mbox{ and }\quad
 \Delta_{1,P}=
  \begin{bmatrix}
  -2 &  1 &  \cdots &  0 &  1\\
   1 & -2 & \cdots &  0 &  0 \\
   \vdots & \vdots & \ddots  &  \vdots & \vdots \\ 
   0 &  0 &  \cdots & -2 &  1 \\ 
   1 &  0 &  \cdots &  1 & -2 \\ 
 \end{bmatrix},
  \end{equation}
respectively. 
 
In the $d$-dimensional setting we have the similar Kronecker rank-$d$ representations.
For example, in the case  $d=3$ the "periodic"  $N_d\times N_d$ Laplacian matrix 
$A_{\Delta}$ takes a form
\begin{equation} \label{eqn:Lapl_Kron3}
A_{\Delta} = A_{1,P} \otimes I_2\otimes I_3 + I_1 \otimes A_{2,P} \otimes I_3 + I_1 \otimes I_2\otimes A_{3,P},
 \end{equation} 
such that its Kronecker rank equals to $3$, while for the arbitrary $d\geq 3$, we have $rank_{Kron}(A)=d$.

\subsection{Fast matrix assembling for the stochastic part}\label{ssec:FDM_Kron_stoch} 
 
The Kronecker form representation of the "stochastic" term in (\ref{eqn:FEM_discr})
 further denoted by $\widehat{A}_s$ is more involved. 
 For the ease of exposition we, first, discuss the case $d=2$, and assume that $n_1=n_2$.
 
 For given stochastically chosen distribution 
 of non-overlapping cells $S_k$, $k=1,\ldots,K$, where the constant coefficient is perturbed,
 we introduce  the full covered  grid domain 
 $\widehat{G} = \cup^{K}_{k=1} S_k \subset \Omega$ colored by black in Figure \ref{fig:2D_example_Check}  
 and \ref{fig:Clust_vsL_Checkeboadr}. 
   We obtain a union
 of non-overlapping ``covered'' square cells $S_k$, $k=1,\ldots,K$, $K\leq L^2$, each of the grid-size  
 $\overline{n}_0\times \overline{n}_0$,
 \begin{equation}\label{eqn:NO-decomp}
  \widehat{G}= \cup^{K}_{k=1} S_k, \quad S_k \subseteq G_k,
 \end{equation}
 where the number $K$ varies for different realizations.
 By construction, we have $a(x)=1$ for $x\in \widehat{G}$ and $a(x)=\lambda$ for 
 $x\in \Omega  \setminus \widehat{G}$.
 Here $\overline{n}_0= 2^p+1$, for some
  $p=1,2,\ldots p_0$, is fixed as above by the chosen overlap constant $\alpha >0$, 
  see \S\ref{ssec:Discret_Scheme}.
 % relation $n_1 -1 =(m_s -1) 2^p$.
 In this construction, the non-overlapping elementary cells $S_k$ for different $k$ are allowed to
 have the only common edges of size $\overline{n}_0$. 
 
 Notice that in the case of non-overlapping
 decomposition (\ref{eqn:NO-decomp}) the set of cells $\{S_k\}$ may coincide with  
 the initial set $\{G_s\}$ that allows to maximize the  size $\overline{n}_0\times \overline{n}_0$ 
 of each $S_k$, $k=1,\ldots,L^2$, 
 to the largest possible, i.e. to $\overline{n}_0= n_0 +1$. 
 We refer to \cite{KKO:17} for the construction in the case  of general overlapping coefficients.
  
 To finalize the matrix generation procedure for $\widehat{A}_s$, we define the local 
 $\overline{n}_0\times \overline{n}_0$
 matrices representing the discrete Laplacian with Neumann boundary conditions,
 \[
  \widehat{Q}_{\overline{n}_0}:= \mbox{tridiag} \{ 1,-2,1 \} + \mbox{diag}\{ 1,0,\ldots,0,1 \}
  \in \mathbb{R}^{\overline{n}_0\times \overline{n}_0},
 \]
and the diagonal matrix 
$$
\widehat{I}_{\overline{n}_0}:=\mbox{diag}\{ 1/2,1,\ldots,1,1/2 \}
\in \mathbb{R}^{\overline{n}_0\times \overline{n}_0},
$$
see the visualization in (\ref{BC_Neu_peri}), left. 
Here, we may select $\overline{n}_0=3,5,...$ that corresponds to the choice $p= 1,2,...$.
In the case of $\overline{n}_0 \times \overline{n}_0 $ matrix with minimal size 
$\overline{n}_0=3 $,  both discrete Laplacians in (\ref{BC_Neu_peri}) simplify to
\begin{equation}\label{matr_mini}
   \Delta_N=   
   \begin{bmatrix}
  -1 &  1 & 0  \\
   1 & -2 & 1  \\
   0 &  1 &  -1 \\ 
 \end{bmatrix}
 \quad \mbox{ and }\quad
      \Delta_P= 
      \begin{bmatrix}
  -2 & 1 &   1\\
   1 & -2 & 1  \\
   1 & 1 &  -2\\ 
 \end{bmatrix}.
 \end{equation}
 
 Let the subdomain $S_k$ be supported by the index set $I_k^{(1)}\times I_k^{(2)}$ 
of size $\overline{n}_0\times \overline{n}_0$ for $k=1,\ldots,K$.
Introduce the block-diagonal matrices $\overline{Q}_k \in \mathbb{R}^{n_1\times n_1}$
and $\overline{I}_k \in \mathbb{R}^{n_1\times n_1}$ by inserting matrices 
$\widehat{Q}_{\overline{n}_0}$ and $\widehat{I}_{\overline{n}_0}$, defined above,  
as diagonal blocks into $n_1\times n_1$ zero matrix in the positions 
$I_k^{(1)}\times I_k^{(1)}$ and $I_k^{(2)}\times I_k^{(2)}$, respectively.
Now the stiffness matrix $\widehat{A}_s$ is represented in the form of a Kronecker product sum as follows,
\begin{equation}\label{eqn:As_Kron}
 \widehat{A}_s= \sum^{K}_{k=1} (\overline{Q}_k \otimes \overline{I}_k + 
 \overline{I}_k \otimes \overline{Q}_k) + P^{(2)},
\end{equation}
where 
$$
P^{(2)}= P^{(1)} \otimes I_{n_1} + I_{n_1} \otimes P^{(1)} \in \mathbb{R}^{N_d\times N_d}
$$ 
is the "periodization" matrix in 2D.

In a $d$-dimensional case the representation (\ref{eqn:As_Kron}) generalizes 
to a sum of $d$-factor Kronecker products
\begin{equation}\label{eqn:As_Kron_d}
 \widehat{A}_s= \sum^{K}_{k=1} (\overline{Q}_k \otimes \overline{I}_k\otimes\cdots\otimes \overline{I}_k+
\ldots + \overline{I}_k \otimes \cdots \otimes \overline{I}_k\otimes \overline{Q}_k)
 + P^{(d)},
\end{equation}
where $P^{(d)}$ is the "periodization" matrix in $d$ dimensions, constructed 
as the $d$-term Kronecker sum similar to the case $d=2$.

The Kronecker product form of  (\ref{eqn:As_Kron}) and  (\ref{eqn:As_Kron_d})  leads  to
the corresponding Kronecker sum representation for the total stiffness matrix $A$. 
This allows an efficient implementation of 
the matrix assembly and low storage request for the stiffness matrix preserving the Kronecker sparsity. 
In general, for 2D case the number $K$ of elementary cells\footnote{For example, 
for cells of minimal size, $\overline{n}_0  \times \overline{n}_0$ 
with $\overline{n}_0=3$, as in (\ref{matr_mini}), we have $K=O(n_1^2)$.} 
does not exceed  $L^2$, and it may coincides with $L^2$
only in the case of non-overlapping decomposition $\widehat{G} = \cup^{L^2}_{k=1} S_k$
with maximal size $\overline{n}_0=n_0 +1$,
where different patches $S_k$ are allowed to have joint pieces of boundary, but no overlapping area.

The technical assumption that all sub-cells $S_k$ are supposed to be cell-centered is not 
essential for the presented construction. The approach also applies to the case of general 
location of $S_k$ inside of the corresponding unit cell $G_k$.

For the above constructions, which apply to any dimension $d$,
we are able to prove the following storage complexity and Kronecker rank estimates 
for the stiffness matrix $A$.
\begin{lemma}\label{lem:Matr_stor_Krank}
 The storage size for the stiffness matrix $A$  is bounded by
 \[
 Stor(A)\approx  Stor(\widehat{A}_s) = O(d \overline{n}_0 K + d n_1), \quad K\leq L^d.
 \]
 In the general case $d\geq 2$ the Kronecker rank of the matrix ${A}$ is bounded by
\[
 \mbox{rank}_{Kron} ({A}) \leq   K \leq L^d.
\]
% The Kronecker rank of the stiffness matrix reduces in two cases:\\
% (a) For the case of non-overlapping cells $G_s$, $s=1,\ldots,L^d$, we have 
% \[
% \mbox{rank}_{Kron} ({A}) \leq L^d.
% \]
In the case of cell-centered locations of subdomains $S_k$ 
(special case of geometric homogenization) there holds 
\[
\mbox{rank}_{Kron} ({A}) \leq L^{d-1}. 
\]
\end{lemma}
\begin{proof} 
 The first two estimates directly follow from the construction. To justify the improved rank 
 estimate in 2D case we notice that the $K$-term sum in (\ref{eqn:As_Kron}) can be simplified as follows.
 Introduce the $L$ horizontal grid strips ${\cal S}_\ell$, $\ell=1,...,L$, each of width $n_0$
 and agglomerate all the summands in (\ref{eqn:As_Kron}) with $S_k \subseteq {\cal S}_\ell$ into 
 one matrix $A_\ell$. It can be seen that $\mbox{rank}_{Kron} ({A}_\ell)=1$. Hence the equation
 \[
  A= \sum^{L}_{\ell=1} A_\ell
 \]
proves the result for $d=2$. The rank estimate in the case $d \geq 3$ can be derived completely 
similar.
\end{proof}
 
 The discretized equation (\ref{eqn:PDE_stoch_stand}) takes a form % of (\ref{eqn:FEM_Galerk}), 
 \begin{equation}\label{eqn:FEM_syst1}
 A {\bf u} = {\bf f}, % \quad  \mbox{for}\quad  i=1,\ldots,d,
 \end{equation}
 where the FEM-Galerkin matrix $A$ generated by the equation coefficient 
 ${\mathbb{A}}_n(x)$ is calculated as described above (see \cite{KKO:17} for 
 the case of overlapping decompositions).
\begin{figure}[htb]
\centering
\includegraphics[width=4.6cm]{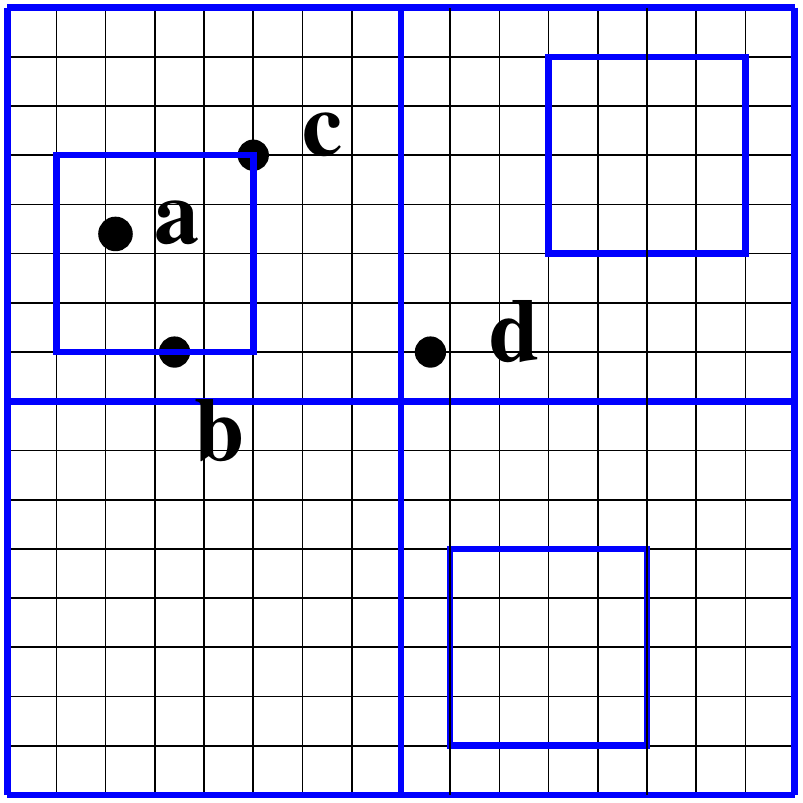} \hspace{2cm} 
\includegraphics[width=4.6cm]{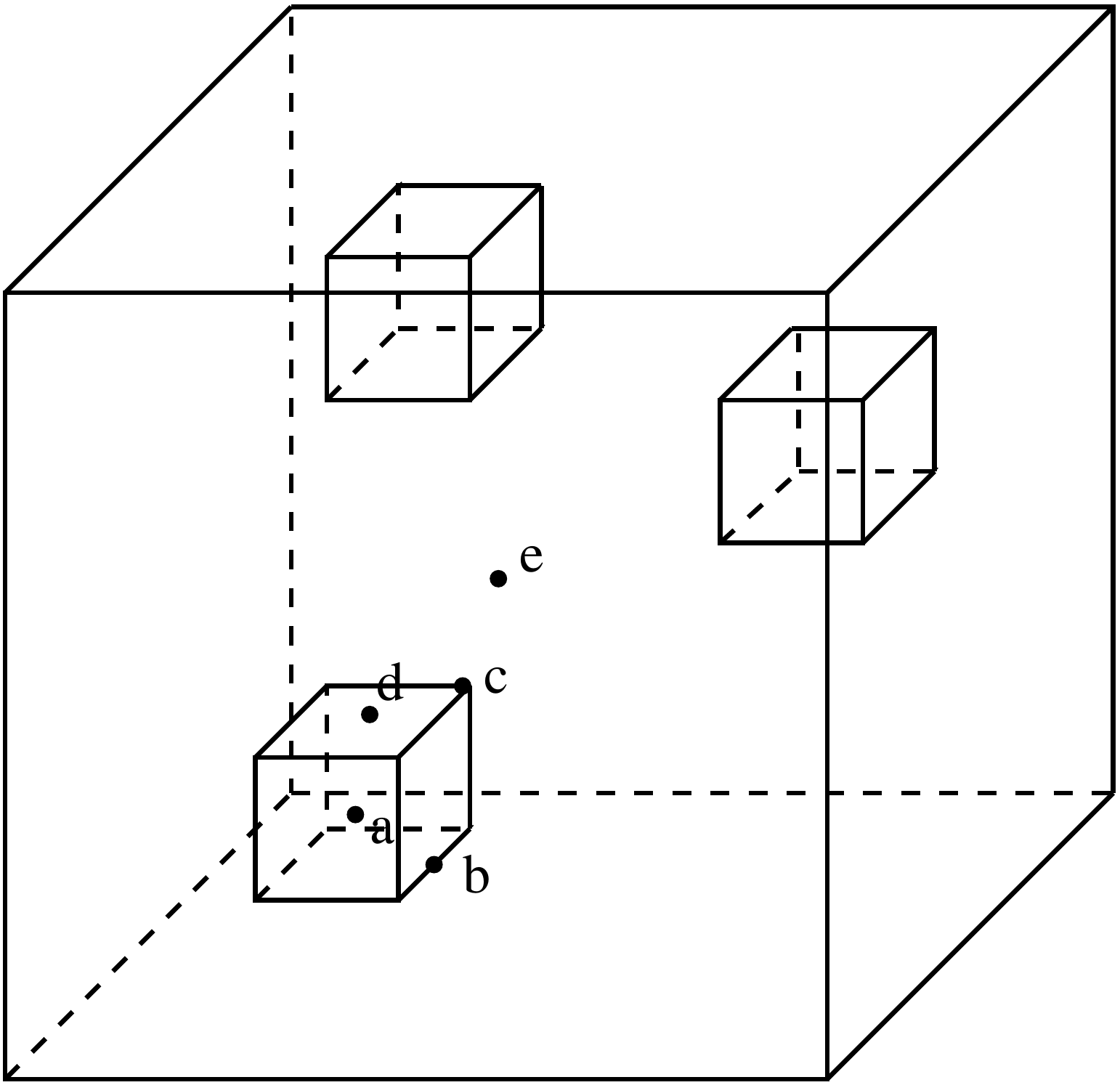}
\caption{\small Accounting jumping coefficients at grid points  (a), (b), (c) and (d) for 
the 2D problem depending on the location from the supports of the bump (left), 
and at points  (a), (b), (c), (d) and (e) for the 3D problem (right).}
\label{fig:Check_coef}
\end{figure}

Notice that in application to RVE approximation of the homogenized matrix one has to solve $d$ linear systems of equations
with different right-hand sides.
The corresponding vector representation ${\bf f}_i\in \mathbb{R}^N$, $i=1,\ldots,d$, of the right-hand side
$f_i(x)$ is computed by scalar multiplication of $f_i(x)$ with the corresponding 
Galerkin basis function and integration by parts, see (\ref{eqn:FEM_discr_RHS}).

Specifically, given the grid-point $x_h\in \Omega_h$, the corresponding  value of 
% $\mathbb{A}_h(x_h)$ 
 the diagonal coefficient is defined by $a_m(x_h)$, see (\ref{eqn:coef_diag}).
In case $d=2$ the variable part $\widehat{a}_m(x_h)$, describing the jumping coefficient, is assigned
by $1$ for interior points in $\widehat{G}$,  
by $1/2$ for interface points (the angle equals to $\pi/2$), 
%by $3/4$ for the "interior" L-shaped corners (the angle equals to $3\pi/4$) 
and by $1/4$ for the "exterior" corner of $\widehat{G}$ 
(the angle equals to $\pi/4$), see points (d), (b), (c) and (a) in Figure \ref{fig:Check_coef} (left), 
respectively. Figure \ref{fig:Check_coef} (left) corresponds to $L=2$, the discretization parameter $n_0=8$ 
and the periodic completion of the geometry. 

The corresponding illustration for the 3D case is presented in 
Figure \ref{fig:Check_coef} (right). 
In case of large number of representative volume elements, $L$,  
one observes the complicated interface defining the strongly jumping coefficients.

 \subsection{The RVE approximation of homogenized matrix}\label{ssec:Def_HomoMatr}

For given size of the RVE, $L$, we consider the sequence of problems 
for stochastic realizations specifying the variable part 
 in the $d\times d$ coefficient matrix  $\widehat{\mathbb{A}}^{(m)}(x)$,  $m=1,\ldots,M$,
 \begin{equation}\label{eqn:RHS_grad_2}
 -\lambda \Delta\phi_i -\beta\nabla \cdot \widehat{\mathbb{A}}^{(m)}(\cdot)
 ({\bf e}_i +\nabla \phi_i )=0,
  %-\nabla \cdot A(\cdot) \nabla \Phi_i = \nabla \cdot A(\cdot) {\bf e}_i,
 \end{equation}
 for $i=1,...,d$. 

Fixed $L$ and the particular realization $\mathbb{A}^{(m)}_L(x)= \lambda \, I + 
\beta\widehat{\mathbb{A}}^{(m)}(x) $, 
the averaged coefficient matrix 
$\bar{\mathbb{A}}^{(m)}_L =[\bar{a}^{(m)}_{L,ij}]\in \mathbb{R}^{d\times d}$, 
$i,j=1,\ldots,d$, with the constant entries is conventionally defined by the equation
\begin{equation}\label{eqn:A_hom_m}
 \bar{\mathbb{A}}^{(m)}_L  {\bf e}_i =\int_\Omega \mathbb{A}^{(m)}_L (x)({\bf e}_i + 
 \nabla \phi_i ) dx,
\end{equation}
which implies the representation for matrix elements
%\begin{equation*}\label{eqn:A_hom_m}
\[
 \bar{a}^{(m)}_{L,ij} \equiv \bar{a}^{(m)}_{ij} = %\lambda \, \delta_{ij} + 
 \int_\Omega [(\lambda I_{d \times d} +   \beta\widehat{\mathbb{A}}^{(m)} (x))
 ({\bf e}_i  + \nabla \phi_i )]_j dx,  \quad i,j=1,\ldots,d.
 \]
%\end{equation*}

The latter leads to the entry-wise representation of the homogenized matrix 
$\bar{\mathbb{A}}^{(m)} =[\bar{a}^{(m)}_{i j}]$, $i,j=1,\ldots,d$,
% $\mathbb{A}_{hom,m}=[a_{ij}]$, $i,j=1,\ldots,d$,
\begin{align} \label{eqn:Ahom3_m_entry}     %{lll} 
    &   \bar{a}^{(m)}_{i i} = \int_\Omega a_n(x) \Bigl(\dfrac{\partial \phi_i}{\partial x_i} +1\Bigr) dx
    % = \int_\Omega \Bigl( a_m(x) - \dfrac{\partial a_m(x)}{\partial x_i} \Phi_i \Bigr) dx,
    \quad i=1,\ldots,d,\nonumber \\
    & \bar{a}^{(m)}_{i j}= \int_\Omega a_n(x) \dfrac{\partial \phi_i}{\partial x_j}  dx
 % = - \int_\Omega \dfrac{\partial a_m(x)}{\partial x_j} \Phi_i  dx, 
 \quad i,j=1,\ldots,d, \;\; i\neq j.  %\nonumber \\
%     & A_ {2,1}= \int_\Omega a_m(x) \dfrac{\partial \Phi_2}{\partial x_1}  dx
%  = - \int_\Omega \dfrac{\partial a_m(x)}{\partial x_1} \Phi_2  dx,  \nonumber \\
%  % & A_ {2,1} =\int_\Omega a_m(x) \dfrac{\partial \Phi_2}{\partial x_1}  dx,  \quad \quad\qquad  
%  & A_ {2,2}= 
%  \int_\Omega a_m(x) \Bigl(\dfrac{\partial \Phi_2}{\partial x_2} +1 \Bigr) dx
%  = \int_\Omega \Bigl( a_m(x) - \dfrac{\partial a_m(x)}{\partial x_2} \Phi_2 \Bigr) dx. % \noindent
\end{align}

Taking into account (\ref{eqn:coef_diag}) and (\ref{eqn:RHS_A_hom_m}), we obtain the 
computationally convenient representation
\begin{align} \label{eqn:Ahom3_m_entry_comp}     %{lll} 
   \bar{a}^{(m)}_{i i} &= \lambda + \beta\int_\Omega \widehat{a}_m(x)dx -
    \beta\int_\Omega \dfrac{\partial \widehat{a}_m(x)}{\partial x_i} \,\phi_i dx  \nonumber \\
    & =
    \lambda + \beta\int_\Omega \widehat{a}_m(x)dx - \int_\Omega f_i \,\phi_i dx,
%     \int_\Omega \Bigl( a_m(x) - f_i \,\Phi_i \Bigr) dx=
%     \lambda + \beta\int_\Omega \Bigl(\widehat{a}_m(x) - 
%     \dfrac{\partial \widehat{a}_m(x)}{\partial x_i} \,\Phi_i \Bigr) dx,
    \quad i=1,\ldots,d, \nonumber \\
   \bar{a}^{(m)}_{i j}  & = - \beta\int_\Omega \dfrac{\partial \widehat{a}_m(x)}{\partial x_j} \phi_i  dx=
    -\int_\Omega f_j \, \phi_i  dx, 
    \quad i,j=1,\ldots,d, \;\; i\neq j.  %\nonumber \\
\end{align}
% The representation (\ref{eqn:Ahom3_m_entry_comp}) proves the symmetry in the matrix $\mathbb{A}_{hom,m}$
% by, first, multiplication of equation (\ref{eqn:PDE_stoch_stand}) with $f_j(x)$ 
% and then integrating by parts. 

The representation (\ref{eqn:Ahom3_m_entry_comp}) ensures the symmetry of the homogenized matrix 
$\bar{\mathbb{A}}^{(m)}$, i.e. $\bar{a}^{(m)}_{ij}=\bar{a}^{(m)}_{ji}$, 
taking into account the equations (\ref{eqn:RHS_grad}), 
see \cite{KKO:17}, \S4.2, for the detailed argument. 
  
In numerical implementation, we apply the same variational scheme for calculation (by numerical integration)
of $\bar{\mathbb{A}}^{(m)}$  as for the FEM discretization 
of the operator and the right-hand side in the initial elliptic PDE, thus preserving the symmetry of the matrix 
$\bar{\mathbb{A}}^{(m)}$ inherited from the initial variational formulation.
 
Furthermore, integrals over $\Omega$ in (\ref{eqn:Ahom3_m_entry_comp}), 
as they are written for the exact matrix entries $\bar{a}^{(m)}_{ij}$,  $i,j=1,\ldots,d$,
are calculated (approximately) by using the discrete representation of integrand
on the grid $\Omega_h$.

\section{Construction of preconditioner for PCG iteration}\label{sec:PCG} 

\subsection{Spectral equivalent preconditioner}\label{ssec:Precond}
 
 Let the right-hand side in (\ref{eqn:hm_setting}) satisfy $\langle {\bf f}, 1 \rangle=0$, 
 then for a fixed realization $n$, the equation 
 \begin{equation}\label{eqn:DD_Eqn}
  A_m {\bf u} =(\lambda A_\Delta + \beta A_{s,m}){\bf u} = {\bf f},
 \end{equation}
 where $A_\Delta=\Delta_h$ is the periodic Laplacian, has the unique solution. 
 We solve this equation by the preconditioned conjugate gradient (PCG) 
iteration   %(routine \emph{pcg} in Matlab library) 
with the preconditioner obtained as inverse of the (perturbed) periodic Laplacain 
\[
 B_\Delta = \frac{1+\lambda}{2} A_\Delta + \delta I, 
\]
where  $I$ is the $N_d\times N_d$ identity matrix and
$\delta \geq 0$ is a small regularization parameter introduced for  
stability reasons in the case of direct inversion of $B_\Delta$ (that is too costly in 3D case). 
In what follows, for 3D case, we set up $\delta = 0$
and use the explicit low Kronecker rank approximation of the 
pseudo-inverse matrix $B_\Delta^{+}$ as the preconditioner for   
solving the algebraic system of equations (\ref{eqn:DD_Eqn}) 
on the subspace $\langle{\bf u},1 \rangle =0$. 
Notice that by the conventional definition we have $B_\Delta^{+} B_\Delta {\bf u} ={\bf u} $
for all ${\bf u}$ such that $\langle{\bf u},1 \rangle =0$, while $B_\Delta^{+} 1 =0$.
Since the pre-factor $\frac{1+\lambda}{2}$ does not effect the condition number of 
the preconditioned matrix, in the following discussion we set up it as $1$. 

It can be proven that the condition number of preconditioned matrix is uniformly 
bounded in $n_1$ and $L$. The following Lemma proves the 
spectral equivalence of the preconditioner, see also \cite{KKO:17}.
\begin{lemma}\label{lem:Cond_Number}
 Given the matrix $B_\Delta$ with $\delta=0$, then for any stochastic realization
 the condition number of the preconditioned matrix $B_\Delta^{+} {A}_n$ 
 on the kernel of the stiffness matrix ${A}_n$, $\langle{\bf u},1 \rangle =0$, 
is uniformly bounded in $n_1$, $L$ and $\alpha$, such that
\[
 cond \{ B_\Delta^{+}{A}_m \} \leq  C{\lambda}^{-1}.
\]
\end{lemma}
 \begin{proof}
 The particular bound on the condition number in terms of a parameter $\lambda$
can be derived by introducing the average coefficient
$$
{a}_0(x)=\frac{1}{2}(a^+(x) + a^-(x)),
$$ 
where $a^+(x)$ and $a^-(x)$ are chosen as {\it majorants and minorants} 
of $a_m(x)$ in (\ref{eqn:coef_diag}), respectively.
%Lemma 4.1 in \cite{BokhSRep:15} shows that 
Indeed, the preconditioner $B^{+}_0$ generated by the coefficient
${a}_0(x)=\frac{1}{2}(a^+(x) + a^-(x))$ allows the condition number estimate
\[
 cond \{B^{+}_0 {A}_n \} \leq C \max\frac{1+q}{1-q}, \quad\mbox{with}\quad
 q:=\max(a^+(x) - {a}_0(x))/{a}_0(x)<1
\]
on the subspace $\langle{\bf u},1 \rangle =0$.
 With the choice $\delta=0$, the preconditioner $B_\Delta^{+}$ corresponds to  $a^+(x)=1$ 
 and $a^-(x)=\lambda$,
 hence, we obtain ${a}_0(x)=\frac{1+\lambda}{2}$ and the result follows.
 \end{proof}
 
The PCG solver for the system of equations (\ref{eqn:FEM_syst1}) with the pseudo-inverse of 
the discrete Laplacian 
as the preconditioner demonstrates robust convergence with the rate $q < 1$ 
almost uniformly in the model and discretization parameters  $L, \lambda, \alpha$ and the grid size  $n_1$.

 Recall that the periodic Laplacian obeys the Kronecker rank-$d$ representation, 
 that is the $d$-level circulant matrix, and, hence, it can be diagonalized by the Fourier transform. 
 In particular, in the 2D case the periodic Laplacian takes form (\ref{eqn:Lap_Kron}).
 In the case  $d=3$, the $N_d\times N_d$  "periodic" Laplacian matrix $A_{\Delta}$ 
 is the three-term Kronecker sum (\ref{eqn:Lapl_Kron3})
and similar for the $d$-term representation in the general case $d\geq 3$.
 
 In the rest of this section, we discuss the main details of our reconditioned PCG scheme for periodic 
 setting  that relies on the 
 low Kronecker rank approximation of the pseudo-inverse matrix $B_\Delta^{+}$. 
 In this scheme the CG iteration applies to the preconditioned system of equations in
 (\ref{eqn:DD_Eqn}) as follows
 \begin{equation}\label{eqn:PCG_DD_Eqn}
  B_\Delta^{+} A_n {\bf u} = B_\Delta^{+}{\bf f},
 \end{equation}
 that is solved on the subspace $\langle{\bf u},1 \rangle =0$ due to periodic setting.
 
 Notice that in the case of
 homogeneous Dirichlet problem such a scheme with the choice $B_\Delta^{-1}=A_\Delta^{-1}$ was 
 described in \cite{HeidKh2Sch:18}. In \S\ref{ssec:PCG_lowrank} we modify this
 construction to the periodic setting.

\subsection{Low Kronecker rank approximation  of the preconditioner}\label{ssec:PCG_lowrank}

The rank-structured preconditioner presented here for periodic setting is obtained by a modification of the
construction described in \cite{HeidKh2Sch:18} for the Dirichlet boundary conditions, 
see also \cite{SchmittKh2Sch:20} where the case of anisotropic Laplacian is considered.

In the periodic setting, the matrix $A_{\Delta}$ can be diagonalized in the $d$-dimensional Fourier basis, 
which implies the factorization (for example, in 2D case)
\begin{equation}\label{eqn:DiagLaplace2D}
\begin{split}
A_{\Delta} &=  (F_1^*\otimes F_2^*) (\varLambda_1\otimes I_2) (F_1\otimes F_2)
	+(F_1^*\otimes F_2^*) (I_1\otimes \varLambda_2) (F_1\otimes F_2)\\
	&= (F_1^*\otimes F_2^*) \varLambda
%	\underbrace{\big( \varLambda_1\otimes I_2 + I_1\otimes \varLambda_2 \big)}_{\eqqcolon \varLambda}
	(F_1\otimes F_2 ),
\end{split}
\end{equation}
with the diagonal matrix 
\[
\varLambda:= \varLambda_1\otimes I_2 + I_1\otimes \varLambda_2\in \mathbb{R}^{n_1^2 \times n_1^2}.
\]
Furthermore, in 3D case we have
\begin{equation}\label{eqn:DiagLaplace3D}
\begin{split}
A_{\Delta} =  &(F_1^*\otimes F_2^*\otimes F_3^*) (\varLambda_1\otimes I_2\otimes I_3) (F_1\otimes F_2 \otimes F_3)\\
	&+(F_1^*\otimes F_2^*\otimes F_3^*) (I_1\otimes \varLambda_2\otimes I_3) (F_1\otimes F_2 \otimes F_3)\\
	&+(F_1^*\otimes F_2^*\otimes F_3^*) (I_1\otimes I_2\otimes \varLambda_3) (F_1\otimes F_2 \otimes F_3)\\
	=\; &(F_1^*\otimes F_2^*\otimes F_3^*) \varLambda
	%\underbrace{\big( \varLambda_1\otimes I_2\otimes I_3 + I_1\otimes \varLambda_2\otimes I_3 + 
	%I_1\otimes I_2\otimes \varLambda_3 \big)}_{\eqqcolon \varLambda}
	(F_1\otimes F_2 \otimes F_3),
\end{split}
\end{equation}
with the diagonal matrix $\varLambda$ given by
\begin{equation}\label{eqn:Lamda3D}
 \varLambda:= \varLambda_1\otimes I_2\otimes I_3 + I_1\otimes \varLambda_2\otimes I_3 + 
	I_1\otimes I_2\otimes \varLambda_3 \in \mathbb{R}^{n_1^3 \times n_1^3}.
\end{equation}
 Here the  $n_1 \times n_1$ diagonal matrices 
 $\varLambda_\ell=\mbox{diag}\{F_\ell({\bf p})\} \in \mathbb{R}^{n_1 \times n_1}$, $\ell=1,\ldots,d$, 
 are defined by the Fourier transform $F_\ell({\bf p})\in \mathbb{R}^{n_1}$ of the first column vector 
 $$
 {\bf p}= (2,-1,0,\ldots,0,-1)^T\in \mathbb{R}^{n_1}
 $$ 
 in the circulant matrix $\Delta_P$  in (\ref{BC_Neu_peri}), diagonalized by the Fourier transform
 \begin{equation}\label{eqn:Circ_Fourier}
  \Delta_P = F_\ell^* \, \mbox{diag} \{ F_\ell({\bf p}) \}F_\ell  = 
  F_\ell^* \, \mbox{diag} \{ \lambda_1,\ldots, \lambda_{n_1} \}F_\ell.
 \end{equation}
 Here  $\lambda_i= [F_\ell({\bf p})]_i$, $i=1,\ldots,n_1$,
 denote the eigenvalues of the periodic Laplacian $\Delta_P$, such that $\lambda_1=0$.
 The latter property introduces the significant difference between the periodic case and 
 the case of Dirichlet boundary conditions considered in \cite{HeidKh2Sch:18}.
 
 The representations \eqref{eqn:DiagLaplace2D} and (\ref{eqn:DiagLaplace3D}) 
 give rise to the eigenvalue decomposition of $A_{\Delta}$. % which has zero eigenvalue. 
 Therefore, for a function $\mathcal{F}$ applied to the matrix $A_{\Delta}$, we arrive at  
\begin{equation}\label{eqn:DiagFLaplace2D}
\mathcal{F}(A_{\Delta})  = (F_1^*\otimes F_2^*) \mathcal{F}(\varLambda) (F_1\otimes F_2 ),
\end{equation}
and 
\begin{equation}\label{eqn:DiagFLaplace3D}
\mathcal{F}(A_{\Delta})  = (F_1^*\otimes F_2^*\otimes F_3^*) 
	\mathcal{F}(\varLambda) (F_1\otimes F_2\otimes F_3 ),
\end{equation}
for 2D and 3D cases, respectively. In our application we have to approximate the diagonal 
matrix $\mathcal{F}(\varLambda)=\varLambda^{-1} $ in the case of periodic Laplacian, where
$A_{\Delta}$ has zero eigenvalue.

To construct the efficient preconditioner for the system matrix 
$A_n$, we are particularly interested in the low Kronecker rank 
approximation of the pseudo-inverse matrix $A_{\Delta}^+$ 
that provides the spectrally close approximation to $\mathcal{F}(A_n)=A_n^{+}$. 
Since in the case of periodic boundary conditions the matrix $A_n$
has zero eigenvalue corresponding to the constant vector, 
one has to consider the pseudo-inverse matrix $A_{\Delta}^+$ instead of 
the standard inverse $A_{\Delta}^{-1}$ that means the matrix $\varLambda^{+}$
should be defined by setting the first element in $\varLambda^{-1} $ to zero.

We discuss two different solution strategies for multiple solving the target 
 algebraic system of linear equations in  (\ref{eqn:DD_Eqn}). 
 
 {\it The first approach} is based on the FFT diagonalization  
 of pseudo-inverse preconditioning matrix $A_\Delta^+$ by using the pseudo-inverse $\varLambda^{+}$. 
 In this case the system matrix 
 is stored in the standard sparse matrix format and the PCG solution process is performed 
 in the exact matrix arithmetics.
 
 {\it In the second approach}, we
 represent both the system matrix and preconditioner in the low Kronecker rank  form
 and store both large matrices $A_n$ and  $A_\Delta^+$  as a  small set of thin Kronecker factor matrices
 which allows to implement the PCG iteration on the low parametric manifold of low-rank discretized functions.

Here we briefly discuss the second approach.
In the 3D case, the elements of the core diagonal matrix $\varLambda$ in (\ref{eqn:Lamda3D}) 
can be represented as a three-tensor 
$$
{\bf G}=[g(i_1,i_2,i_3)]\in \mathbb{R}^{n_1 \times n_2 \times n_3}, \quad i_\ell\in \{1,\ldots,n_\ell\},
$$
where
\[
 g(i,j,k)=\lambda_{i} + \lambda_{j} + \lambda_{k},
\]
implying that ${\bf G}$ has the exact rank-$3$ decomposition. Here the periodicity conditions
imply $\lambda_{1}=0$ such that $g(1,1,1)=0$.
In the case $d=2$ we have 
the two-term sum representation, $g(i,j)=\lambda_{i} + \lambda_{j} $. We further set $n_1 = n_2 = n_3 =n$.

In the 3D case we consider the low-rank decomposition of the $n\times n \times n$ 
reciprocal core tensor obtained by folding of the pseudo-inverse matrix $\Lambda^+$,
\begin{equation}\label{eq:coreTens}
 {\bf G}^+=[g_+(i,j,k)]\approx [\frac{1}{g(i,j,k)}],    %\quad p=1,2,3,4 
\end{equation}
with entries defined for $i,j,k=1,\ldots,n$, by plugging zero value in the element $g_+(1,1,1)$,
and leaving the rest of the elements in the reciprocal tensor $[\frac{1}{g(i,j,k)}]$ unchanged,
 \begin{equation}  \label{eq:3a}
 g_+ (1,1,1) =0 \quad \mbox{and}\quad
 g_+ (i,j,k) = \frac{1}{\lambda_i +\lambda_j +\lambda_k}\quad \mbox{otherwise}.
 \end{equation} 

The theory on the low-rank tensor approximation of the target tensor ${\bf G}^+$ is based on
the integral Laplace transform representation of the discrete function 
$g_+(i,j,k)=(\lambda_{i} + \lambda_{j} + \lambda_{k})^{-1}$  for $i+j +k >3$
\begin{equation} \label{eqn:Laplace_transf_rhom1}
 g_+(i,j,k)=\frac{1}{\pi} \int_0^\infty e^{- (\lambda_{i} + \lambda_{j} + \lambda_{k}) t}\, dt,
 \quad 
\end{equation}
under the conditions
\begin{equation} \label{eqn:condPositive}
 \lambda_{i} + \lambda_{j} + \lambda_{k} \geq a >0, \quad i,j,k=1,\ldots,n, \quad \mbox{for} \quad i+j +k >3,
\end{equation}
which is satisfied for our construction.

Given $\varepsilon >0$, under the condition (\ref{eqn:condPositive}) there exist 
the rank-$R$ canonical $\varepsilon$-approximation of the tensor 
${\bf U}=[\frac{1}{\lambda_i +\lambda_j +\lambda_k}]$ with $R=O(|\log \varepsilon |)$, 
see for example \cite{KhorBook:17} for more details. 
Then we split the target tensor ${\bf G}^+$ into the sum
\[
 {\bf G}^+= {\bf G}_1 +  {\bf G}_{m-1}, 
\]
where
\[
 {\bf G}_1 = {\bf G}^+ \quad \mbox{for} \quad i=1,\; j,k=1,\ldots,m, \quad \mbox{and} 
 \quad  {\bf G}_1 =0 \quad  \mbox{otherwise},
\]
while 
\[
 {\bf G}_{n-1} = {\bf G}^+ - {\bf G}_1.
\]
Now applying the result for tensors of the type ${\bf U}$ to each of components separately we easily obtain
the rank bounds 
$$
\mbox{rank}({\bf G}_{n-1}) \leq R \quad  \mbox{and}\quad  \mbox{rank}({\bf G}_{1}) \leq R+1,
$$
thus implying the desired upper bound on the canonical rank 
$$
rank({\bf G}^+) \leq 2R +1.
$$
This argument proves the following
\begin{lemma}\label{lem:RankPinverse}
 Given $\varepsilon >0$, under the condition (\ref{eqn:condPositive}) there exist 
the rank-$R$ canonical $\varepsilon$-approximation of the 
tensor $[\frac{1}{\lambda_i +\lambda_j +\lambda_k}]$. 
The target tensor ${\bf G}^+$ can be approximated by a rank $2R +1$ canonical decomposition
with accuracy $2 \varepsilon $, where $R=O(|\log \varepsilon |)$. 
\end{lemma}

From the algorithmic point of view, the rank-structured approximation to the discrete periodic 
Laplacian inverse operators is performed by using
  the multigrid Tucker decomposition of the 3D tensors obtained by reshaping of the
diagonal system matrix represented in the Fourier basis into a 3D tensor ${\bf G}^+$. The subsequent
Tucker-to-canonical decomposition, see \cite {KhKh3:08}, transforms  the Tucker core tensor 
to a canonical one with a small rank, preserving the approximation precision of the Tucker decomposition.

Now assume that $\mathcal{F}(\varLambda)$ can be expressed approximately by a short-term linear combination 
of Kronecker rank-$1$ matrices. 
Then, the low-rank approximation of  $\mathcal{F}(A)$ is reduced to approximation of the diagonal 
matrix $\mathcal{F}(\varLambda)$.
Assume we have a decomposition (in 2D case)
\begin{equation*}
\mathcal{F}(\varLambda)  = \sum_{k=1}^R \mbox{diag} \big(\mathbf{u}_1^{(k)} \otimes \mathbf{u}_2^{(k)}\big),
\end{equation*}
with vectors $\mathbf{u}_i^{(k)}\in \mathbb{R}^{m_i}$ and $R \ll \min(n_1,n_2)$. Now let
$\mathbf{x}\in \mathbb{R}^{N_d}$ be a vector given in a low-rank format, i.e.
\begin{equation*}
	\mathbf{x} = \sum_{j=1}^S \mathbf{x}_1^{(j)} \otimes \mathbf{x}_2^{(j)},
\end{equation*}
with vectors $\mathbf{x}_i^{(j)}\in \mathbb{R}^{n_i}$ and $S \ll \min(n_1,n_2)$. Then 
a matrix-vector product can be calculated by using only 1D matrix-vector operations
\begin{equation}\label{eqn:Lap2DtimesX}
\begin{split}
	\mathcal{F}(A) \mathbf{x} &= (F_1^*\otimes F_2^*) 
	\bigg( \sum_{k=1}^R \mbox{diag} \big(\mathbf{u}_1^{(k)} \otimes \mathbf{u}_2^{(k)}\big) \bigg)
	(F_1\otimes F_2 )
	\bigg( \sum_{j=1}^S \mathbf{x}_1^{(j)} \otimes \mathbf{x}_2^{(j)} \bigg)\\
	&=
	(F_1^*\otimes F_2^*) 
	\bigg( \sum_{k=1}^R \mbox{diag} \big(\mathbf{u}_1^{(k)} \otimes \mathbf{u}_2^{(k)}\big) \bigg)
	\bigg( \sum_{j=1}^S F_1 \mathbf{x}_1^{(j)} \otimes F_2 \mathbf{x}_2^{(j)} \bigg)\\
	&= \sum_{k=1}^R \sum_{j=1}^S F_1^* \big( \mathbf{u}_1^{(k)} \odot F_1 \mathbf{x}_1^{(j)} \big) 
	\otimes
	F_2^*\big( \mathbf{u}_2^{(k)} \odot F_2 \mathbf{x}_2^{(j)} \big),
\end{split}
\end{equation}
where $\odot$ denotes the componentwise (Hadamard) product of 1D vectors.
Using the FFT, the expression \eqref{eqn:Lap2DtimesX} can be computed in factored 
form in $\mathcal{O}(RS n\log n)$ flops, where
$n = \max(n_1,n_2)$.

In the case $d=3$, equation \eqref{eqn:Lap2DtimesX} takes a form
\begin{equation}\label{eqn:Lap3DtimesX}
	\mathcal{F}(A) \mathbf{x} = 
	\sum_{k=1}^R \sum_{j=1}^S F_1^* \big( \mathbf{u}_1^{(k)} \odot F_1 \mathbf{x}_1^{(j)} \big) \otimes
	F_2^*\big( \mathbf{u}_2^{(k)} \odot F_2 \mathbf{x}_2^{(j)} \big)
	\otimes F_3^*\big( \mathbf{u}_3^{(k)} \odot F_3 \mathbf{x}_3^{(j)} \big),
\end{equation}
and similar in the general case of $d>3$. Notice that the total number of terms in (\ref{eqn:Lap3DtimesX}),
that is $RS$, does not depend on the number of dimensions $d$.

 The numerical performance of the PCG iteration with the rank-structured preconditioner  
$A^+_\Delta$ described above will be discussed in section \ref{sec:Numer_2D_3D}.

%%%%%%%%%%%%%%%%%%%%%%%%%%%%%%%%%%%%%%%%%%%%%%%%%%%%%%%%%%%%%%%%%5

\section{Numerical study}\label{sec:Numer_2D_3D}

\subsection{Fast solvers using matrix generation combined with tensor-product preconditioner}
\label{ssec:Fast_Matr}

Here, we consider the beneficial properties of the presented elliptic problem solver 
applied to the case of checkerboard type coefficients profile compared with the 
previous scheme with more general overlapping-type configurations \cite{KKO:17}.
Figure \ref{fig:Clust_vsL_Checkeboadr} illustrates the configuration of the stochastic 
realization for checkerboard type coefficient with large RVE size, $L=256$.
 \begin{figure}[htb]
\centering
\includegraphics[width=7.6cm]{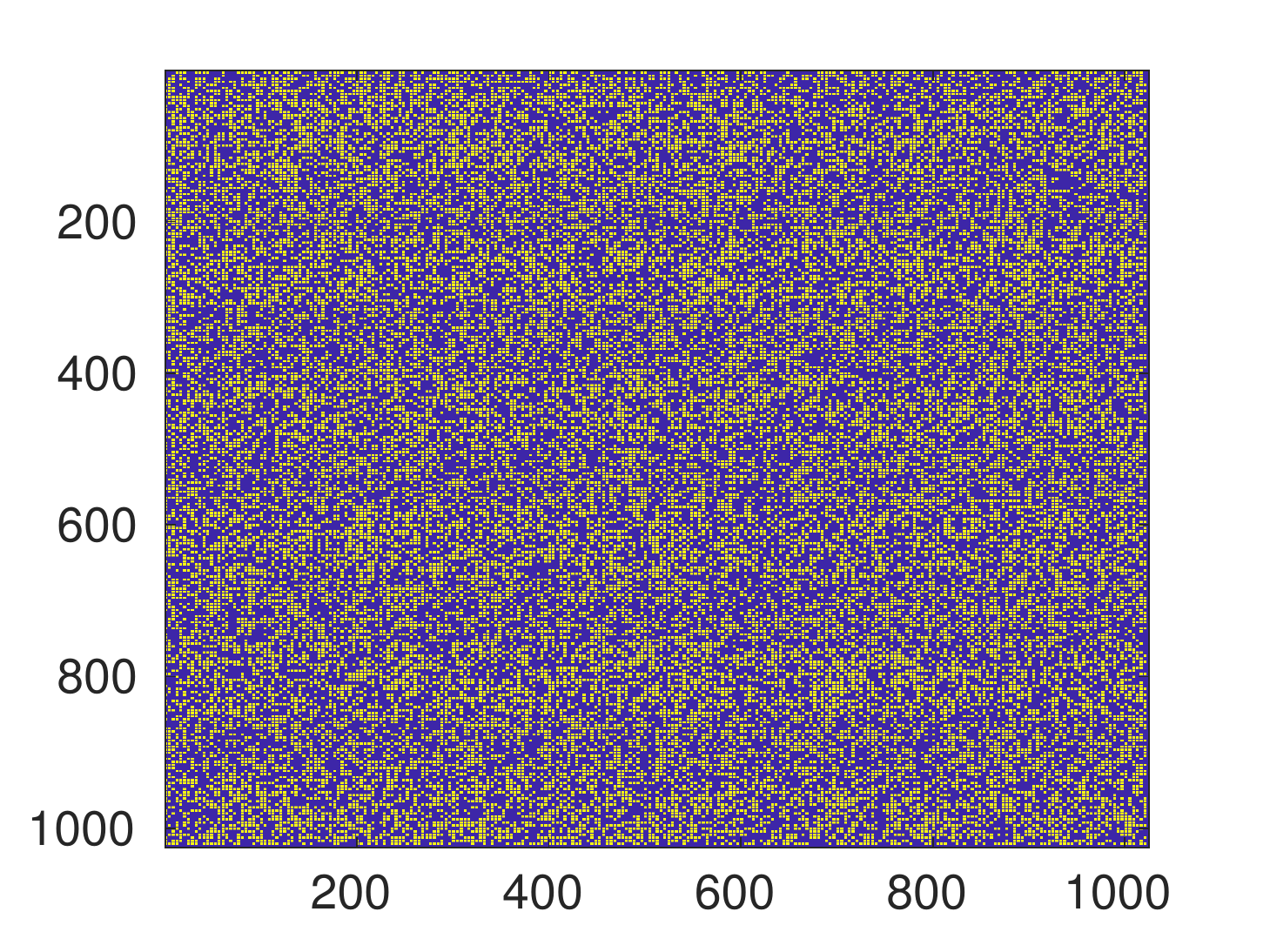}
\includegraphics[width=7.6cm]{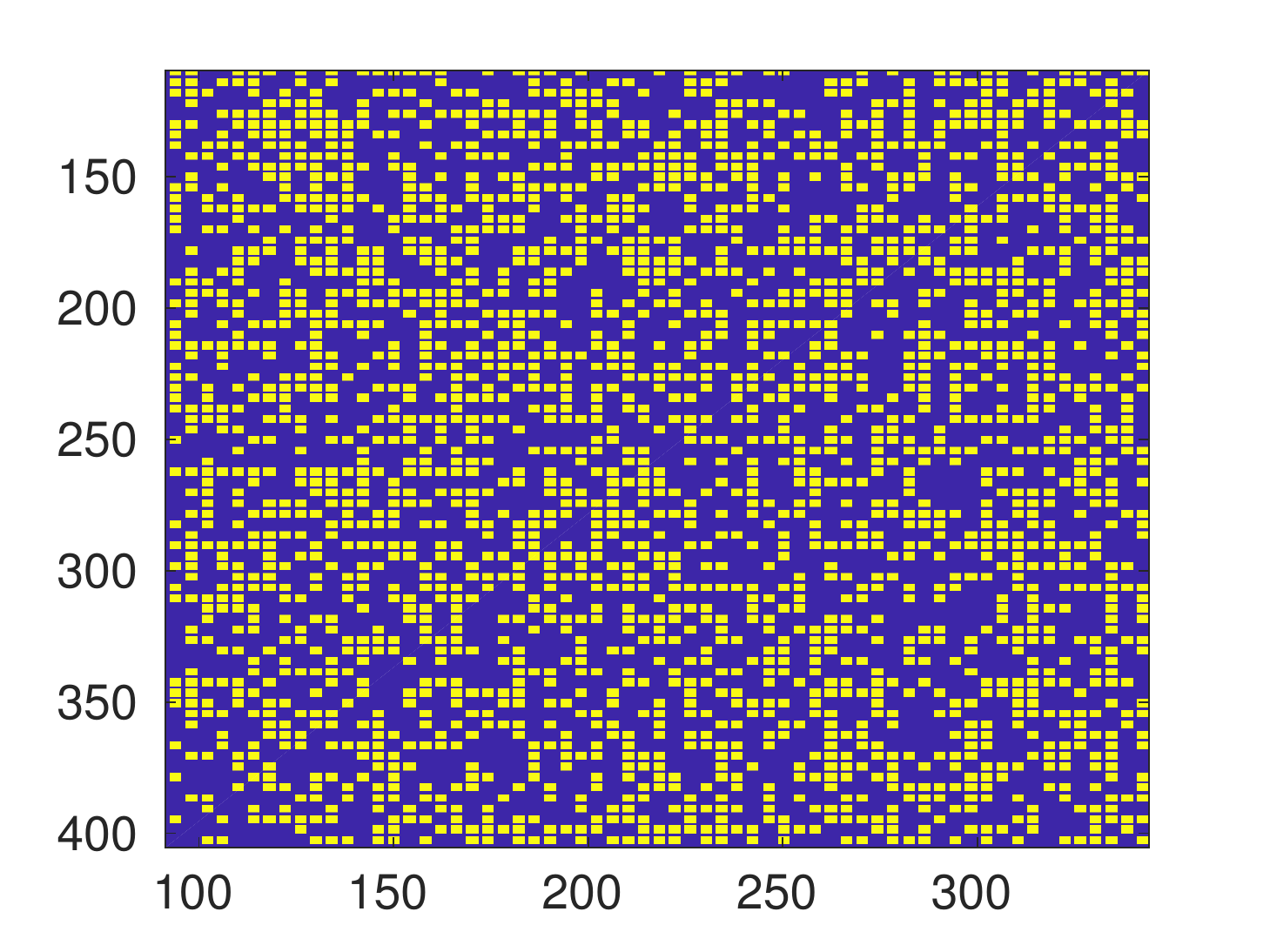}\\
 \caption{\small Example of the realization of coefficient for   $L\times L$ lattice with $L=256$,  
   $\alpha =1/4$, computed on the grid $n\times n$ with $n=1024$. 
   Right panel zooms for the coefficients configuration on the left.}
\label{fig:Clust_vsL_Checkeboadr}
\end{figure}

Table \ref{Tab:Times_check_2D} compares the CPU times (sec.) for matrix generation, calculation 
of the loading vector in the right-hand side, and for the PCG iteration 
 in the course of solving equation (\ref{eqn:DD_Eqn}) for the general case of 
 overlapping-type realizations of coefficients considered 
 in \cite{KKO:17} (marked by (1)),  and for the approach presented in this paper and 
 based on the low-rank preconditioner (marked by (2)).
 Here we set $n_0=4$, $\alpha=1/4$ (see Figure \ref{fig:Clust_vsL_Checkeboadr}, top), 
 such that the vector size $n_1$ on the finest grid is about $n_1 \approx 4.0 \cdot 10^{6}$.
 The iteration stopping criteria is chosen by $\epsilon=10^{-8}$.

\begin{table}[tbh]
\begin{center}\footnotesize
\begin{tabular}
[c]{|r|r|r|r|r|r|r|r||r|}%
\hline
$L^2$ &  $m$/$m^2$ &   matr. (1) &   matr. (2) & RHS (1) &   RHS (2) & solve (1) & solve (2)  \\
% \hline
% 2 /4        & 9/81      & 0.004      & 0.003   & 0.004     & 0.002    & 0.0003   & 0.004 \\
\hline  
$4^2$ &   17/289     & 0.012   &  0.006   & 0.01   &  0.007  &  0.006 & 0.01 \\
\hline
$8^2$ &   33/1089      & 0.06     & 0.007   & 0.045  & 0.003   & 0.137  & 0.027 \\
\hline
$16^2$ &  65/4225      & 0.34     & 0.014  & 0.19   & 0.010  & 0.11  & 0.15\\
\hline
$32^2$ &  129/16641    & 3.0    & 0.038  & 0.8    & 0.014   &  0.5  & 0.6\\
\hline
$64^2$ &  257/66049     & 36     & 0.24  & 3.7    & 0.069  & 2.6  & 2.2  \\
\hline
$128^2$ & 513/263169   & 561    & 1.4   & 22     & 0.38  & 13.8  & 11.9 \\
\hline
$256^2$ & 1025/  $1.0 \cdot 10^6$  & --  & 11.3  &  --     & 2.84  & --  & 66.8 \\
\hline
$512^2$ & 2049/ $4.0 \cdot 10^6$  & -- & 105.6 &  --     & 1.2  & --  & 360.0 \\
\hline
 \end{tabular}
\caption{\small  CPU times (sec) for generating the stiffness matrix for the 2D stochastic 
homogenization problem, the right-hand side, 
and for the solution of the discretized linear system of equations for the checkerboard  
type realizations (2), compared with that for the case of overlapping samples (1).
}
\label{Tab:Times_check_2D}
\end{center}
\end{table}
Solver (1) corresponding to overlapping coefficients profile  
displays limitations in the matrix generation times  (dominating part in the calculations) 
for $L=128$, which exceeds  $561$ seconds.  Matrix generation time  
 for the  solver (2) using checkerboard-type coefficient is essentially smaller ($1.4$ seconds),
which is important taking into account the required large number of realizations 
(up to $M\sim 10^4 \div 10^5$) to be performed for analyzing important quantities
of stochastic problem in random media.
 For larger $L$ the solver (2) might be limited by the solution time, which is about $360$ 
 seconds for $L=512$. 
 
The latter limitation can be relaxed by using the low Kronecker rank (LKR)
preconditioner  $A^+_\Delta$ described in the previous section. 
In this case we observe the balance between the matrix generation cost and that for the PCG iteration.
Table \ref{Tab:times_Fast_solv} represents CPU times (sec.) for PCG iteration
with a regularized preconditioner (RP) $(-\Delta_h + \delta I)^{-1}$
and with the LKR preconditioner by using the rank-structured approximation of the 
pseudo-inverse matrix $A^+_\Delta$, both  applied with parameters 
$n_0=4$, $\alpha=1/4$ and tolerance $\epsilon=10^{-8}$.
The number of inclusions varies from $16$ to $512^2$.
\begin{table}[htb]
\begin{center}\footnotesize 
\begin{tabular}
[c]{|r|r|r|r|r|r|r|r|r|}%
 \hline
 $L^2$   & $4^2$ & $8^2$ & $16^2$ & $32^2$ & $64^2$ & $128^2$ & $256^2$ & $512^2$   \\
  \hline
 RP  & 0.007 & 0.04  & 0.14  & 0.6   & 3.4  & 12.4 & 59.4   & 382.0 \\ 
 \hline
 LKRP   & 0.008 & 0.02  & 0.02  & 0.8   & 0.27 & 0.9 & 6.1  & 11.0 \\ 
 \hline
 \end{tabular}
\caption{\small CPU times for PCG solver with a regularized preconditioner (RP) and with the rank-structured
preconditioner (LKR), $A^+_\Delta$, $n_0=4$, $\alpha=1/4$.}
\label{Tab:times_Fast_solv}
\end{center}
\end{table}

Figure \ref{fig:Times_check_3D_fig} (left) 
 visualizes the data presented in Tables \ref{Tab:Times_check_2D} and \ref{Tab:times_Fast_solv}.
 
 Finally we notice that our numerical tests confirm Lemma \ref{lem:Cond_Number} 
which proves uniform spectral equivalence of the preconditioned in both the grid-size $n_1$ and 
the RVE size $L$. Indeed, fixed the stopping criteria, the number of PCG iterations 
was almost the same for numerical experiments in 2D and 3D cases.

 \subsection{Fast solver for the 3D stochastic problem}
 \label{3D_Numeric}
 
 Numerical simulations for 3D stochastic homogenization require much larger computational 
 resources compared with 2D case.
 Notice that fixed $n_0=4$, the problem size (i.e., vector size $N_d=n^d$) in the 3D case 
 for $L=32$ and $L=64$  equals to $n^3=129^3$ and $n^3=257^3$, respectively, 
 where $n= n_0 \, L +1$. The assembling of the corresponding huge $N_d\times N_d$ system matrix $A$
 is performed by fast tensor based techniques as described in section \ref{ssec:FDM_Kron_stoch}, 
 see also \cite{KKO:17}.
 The system matrix has to be recomputed for large number of stochastic realizations 
 $m=1,\dots,M$, that might be  of the order of several tens of thousand and more. 
 \begin{figure}[tbh]
\centering
\includegraphics[width=7.6cm]{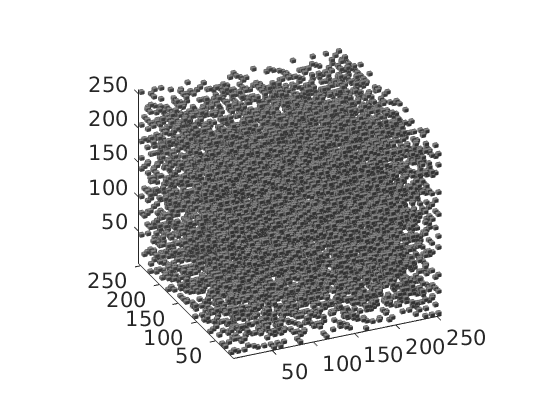}
\includegraphics[width=7.6cm]{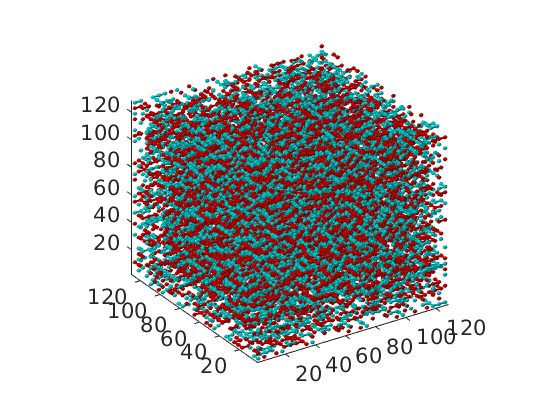}
 \caption{\small Example of 3D stochastic coefficient on $L\times L\times L$ lattice  with  
 $L=32$ with fixed contrast coefficient (left) and the coefficient with two randomly distributed 
 values of contrast.} 
 %$n_0 =8$, and  $\alpha=1/4$.}
\label{fig:3D_coeff_L32}
\end{figure}

Figure \ref{fig:3D_coeff_L32} (left) illustrates the configuration of the matrix coefficient visualized 
 for the 3D realization on the  $L\times L \times L$ lattice, with 
 $L=32$, $n_0 =8$ and the density parameter $\alpha=1/4$. The number of inclusions is 
 about $\frac{1}{2} L^3\approx 16000$. Figure \ref{fig:3D_coeff_L32} (right) presents the example of random
 coefficient with two different values of contrast.
  \begin{table}[htb]
\begin{center}\footnotesize
\begin{tabular}
[c]{|r|r|r|r|l|l|r||}%
\hline
$L^3$ &  $n_1$/$n_1^3$   &  $h$       & $S$ & matr.  & RHS  & solver \\ %(x/y/z)    \\
\hline  
$4^3$ &   17/ 4913    &  6.25e-02 & 32     & 0.012  & 0.013  & 0.04\\ % /0.03/0.03   \\
\hline
$8^3$ &   33/ 35937   & 3.13e-02 & 238     & 0.02  & 0.015  & 0.13 \\ %/0.13/0.13  \\
\hline
$16^3$ &  65/ 274625  & 1.56e-02  & 2048    & 0.1   & 0.05   & 0.62 \\ %/0.62/0.63    \\
\hline
$32^3$ &  129/ 2.1+06   & 7.81e-03 & 16338  & 2.4   & 0.38    & 11.2 \\ %/11.7/11.6    \\
\hline
$64^3$ &  257/ 16.9e+06  & 3.90e-03 & 131494  & 28.9  & 2.7  & 35.1 \\ %/35.1/34.5  \\
\hline
$128^3$ & 513/ 135e+06   & 1.95e-03 & 1049066  & 431.7  & 20.6 & 425 \\ %/417/426      \\
\hline
\end{tabular}
\caption{\small  CPU times (sec) for generating the stiffness matrix for the 3D stochastic 
homogenization problem, the right-hand side , 
and for the solution of the discretized linear system of equations. A vector size of the
corresponding problem is $n_1^3$, and $h$ is the mesh size.
}
\label{Tab:Times_check_3D}
\end{center}
\end{table}

 \begin{figure}[htb]
\centering
\includegraphics[width=7.2cm]{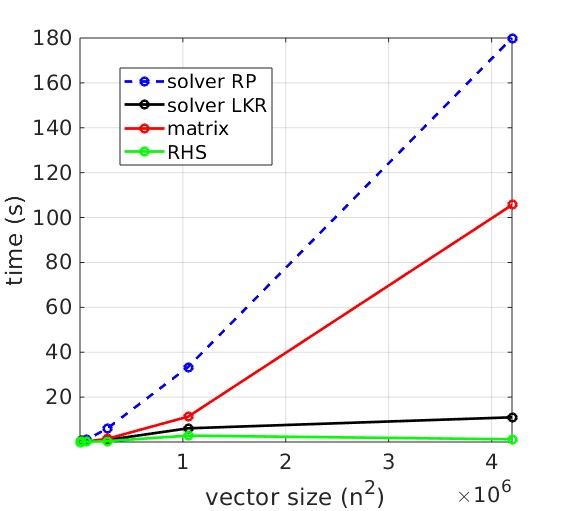} 
\includegraphics[width=7.2cm]{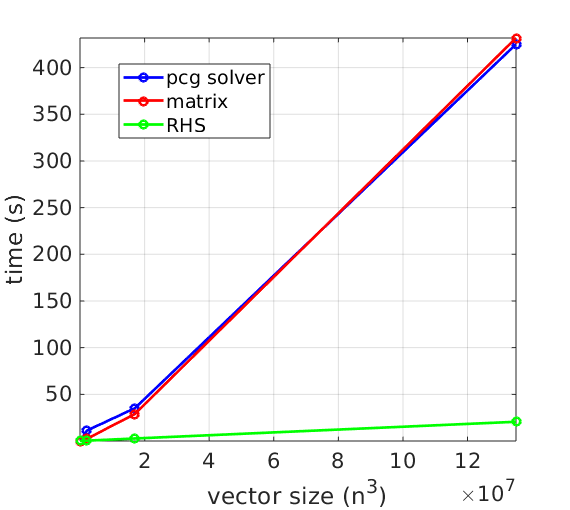}  
\caption{\small Left panel (2D case): times for matrix generation (red line), computation of the RHS (green) and for
 solution of the linear system by the regularized preconditioner (dashed blue) and by the tensor based LKR 
 preconditioner (black) versus the 2D problem size $n^2$.
Right panel (3D case): times for matrix generation (red line), computation of the RHS (green) and for
solution of the 3D problem (blue) versus the system size $n^3$.}
\label{fig:Times_check_3D_fig}
\end{figure}
%  \begin{figure}[htb]
% \centering
% \includegraphics[width=7.0cm]{Times_2D_solver.png}  
% \caption{\small Times for matrix generation (red line), computation of the RHS (green) and for
% solution of the problem by the regularized preconditioner (dashed blue) and by the tensor based LKR 
% preconditioner (black) versus the 2D problem size $n^2$.}
% \label{fig:Times_check_2D_fig}
% \end{figure} 

 We demonstrate the computational complexity of the robust iterative solver for 3D elliptic problems 
applying the stopping criteria $\varepsilon=10^{-7}$. 
In all cases the number of PCG iterations did not exceed $20$ uniformly in the problem size.

Table \ref{Tab:Times_check_3D} displays the computation times for one solve of the 3D stochastic problem
for increasing RVE size of $L^3$, and respectively increasing vector size $n_1^3$ in the linear system 
of equations, as well as the number of stochastic samples, $S$, for the fixed realization 
 %(in some simulations  we solve equations for the sequence of $10^4$ stochastic realizations).
For example, the vector size for $L=128$ is about
$135 \cdot 10^6$, and the corresponding number of stochastic samples exceeds $10^6$.
This huge problem is solved in about $7$ minutes (425 sec) in Matlab.
Figure \ref{fig:Times_check_3D_fig} (right) visualizes the data presented in Table \ref{Tab:Times_check_3D}.

We point out that the main time consuming 
steps include the matrix generation and the solution of the discrete elliptic problem for every of three dimensions 
by calling the ``PCG'' iteration routine with the rank-structured preconditioner in the form of 
periodic Laplacian pseudo inverse. We observe the well balanced complexity of both time consuming steps
of the algorithms.

%%%%%%%%%%%%%%%%%%%%%%%%%%%%%%%%%%%%%%%%%%%%%%%%%%%%%%%%%%%%%%%%%%%%%%%%%%5

\subsection{Example of application to numerical estimation of the homogenized coefficient matrix}
\label{ssec:SystErStandDev}

The set of numerical approximations $\{\bar{\mathbb{A}}^{(m)}_L\}$
to the homogenized matrix $\mathbb{A}_{\mbox{\footnotesize hom}}$ is calculated by 
(\ref{eqn:A_hom_m}) for the sequence $\{\mathbb{A}^{(m)}_L(x)\}$ of $m=1,...,M$ realizations, 
where $M$ is large enough,
and the artificial period $L$ defines the size of RVE. 
For a fixed $L$, the approximation  $\bar{\mathbb{A}}^{M}_L$ is computed 
as the \emph{empirical average}  of the sequence $\{\bar{\mathbb{A}}^{(m)}_L\}_{m=1}^M$,
\begin{equation} \label{eqn:A_hom_Empir}
 \bar{\mathbb{A}}^{M}_L= \frac{1}{M}\sum^{M}_{m=1} \bar{\mathbb{A}}^{(m)}_L.
\end{equation} 
By the law of large numbers we have that the empirical average converges 
almost surely to the \emph{ensemble average} (expectation)
\begin{equation} \label{eqn:A_hom_Ensemble}
 \langle  \bar{\mathbb{A}}_L\rangle_L =  
\lim\limits_{M \to \infty} \bar{\mathbb{A}}^{M}_L. 
%\quad \mathbb{A}_{hom}: = \lim\limits_{L\to \infty}\mathbb{A}_{hom,L}.
\end{equation}
Furthermore, by qualitative homogenization theory, as the artificial period $L\to \infty$,
this converges to the homogenized matrix, see \cite{GlOtto:15},
\begin{equation} \label{eqn:A_hom}
 \mathbb{A}_{\mbox{\footnotesize hom}}: = \lim\limits_{L\to \infty} \langle  \bar{\mathbb{A}}_L\rangle_L .
\end{equation}
In some cases, we use the entry-wise notation for $d\times d$ 
matrices ${\mathbb{A}}=[a_{ij}]$, $i,j=1,\ldots,d$,
for example, $\langle \bar{\mathbb{A}}_L \rangle =[\bar{a}_{L,ij}]$ and 
$\bar{\mathbb{A}}_L^{(m)}=[\bar{a}_{L,ij}^{(m)}]$, etc.

\begin{figure}[htb]
\centering
\includegraphics[width=6.8cm]{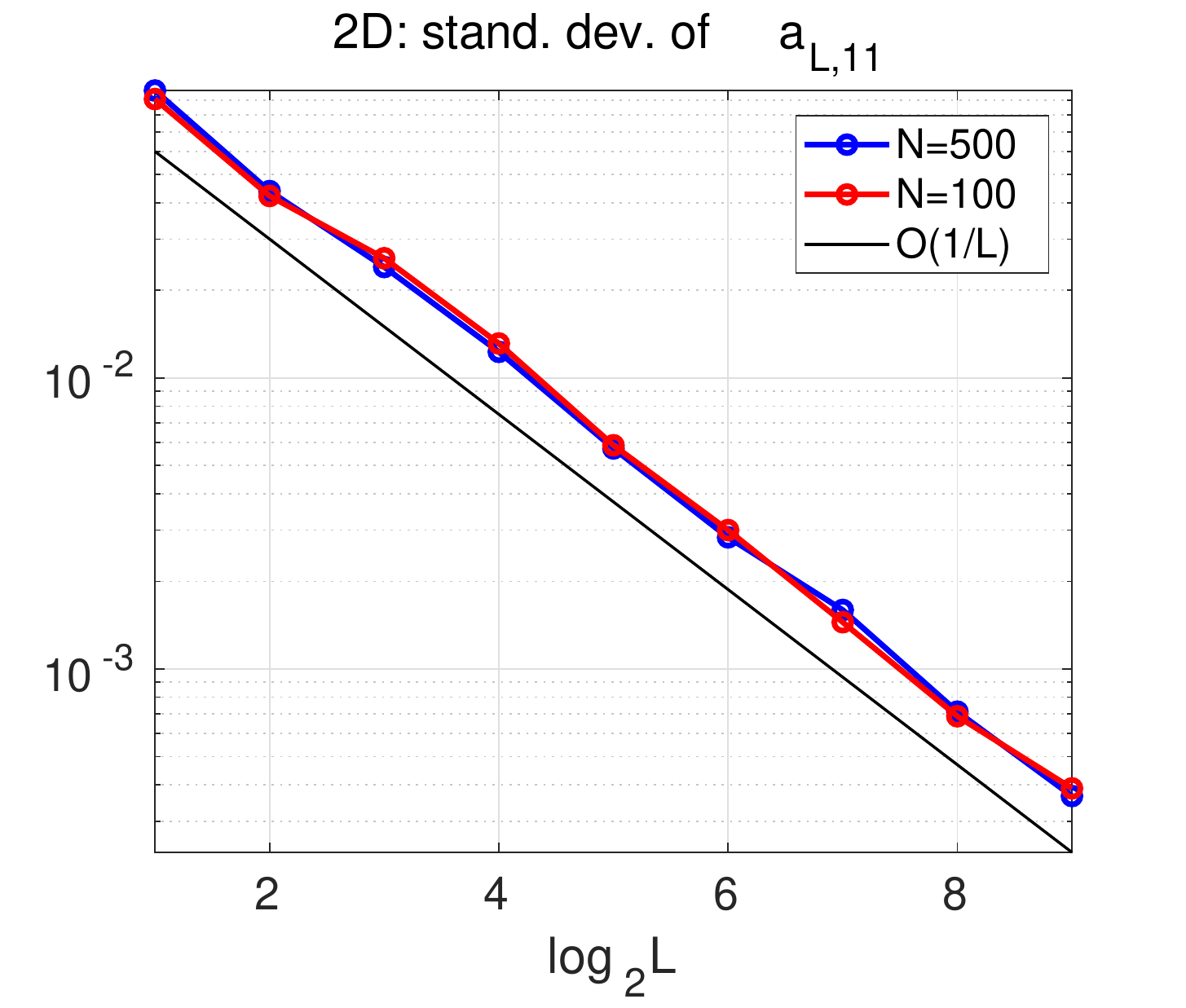}
\includegraphics[width=7.8cm]{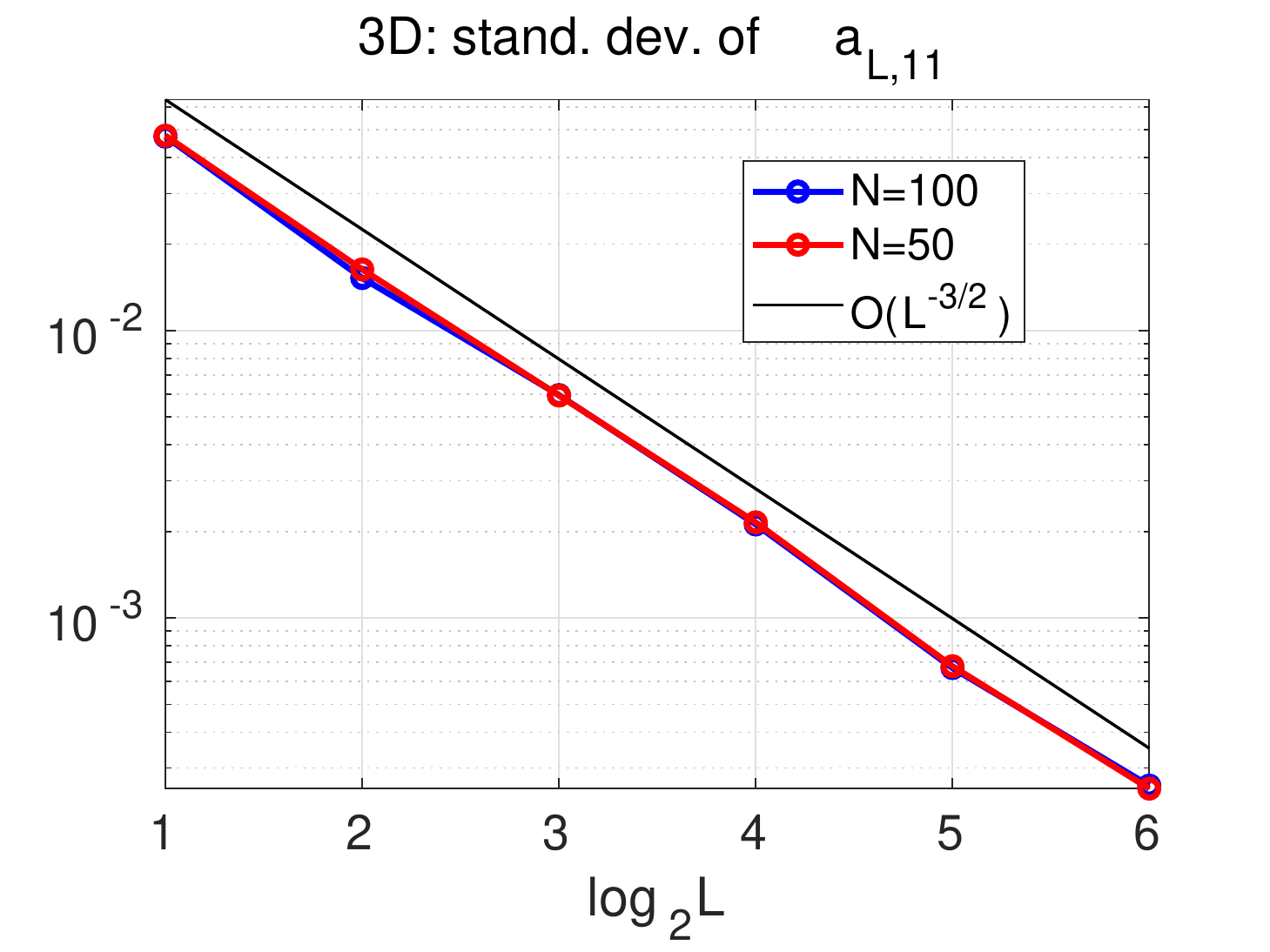}
 \caption{\small Standard deviation for 2D case $a_{L,11}$ versus $L$,  
 $N=500$,  $L=2^p$, $ p=1,\; 2,\, \ldots, 9$ (left) and  
 for 3D case $N=100$ $L=2^p$, $ p=1,\; 2,\, \ldots, 6$ (right);
$n_0=4$, $\alpha=\frac{1}{4}$, $\lambda=0.4$.}
\label{fig:Emp_a12_M20}
\end{figure}

As an example for application of our techniques, 
we  numerically study  the asymptotic of 
the \emph{random part}  of the error (standard deviation) for fixed $N$ of moderate size, 
to confirm the theoretical convergence rate in $L$,
see \cite{GlOtto:15},
 \begin{equation}
\label{eqn:var}
 % var^{1/2}_L (\bar{\mathbb{A}}_L ) = 
  \sigma_L = \langle \left|\mathbb{A}_{\mbox{\footnotesize hom}} -\langle \bar{\mathbb{A}}_L\rangle_L \right|^2 
  \rangle_L^{1/2} \leq C \, L^{-d/2}.
 \end{equation}
In numerical experiments the theoretical value of standard deviation $\sigma_L$ 
is approximated by the commonly used computable quantity $\sigma_L^M$ calculated
for a long enough sequence of $M$ realizations by 
\[
 \sigma_L^M = \sqrt{\frac{1}{M-1}\sum_{m=1}^{M} (\bar{\mathbb{A}}^{(m)}_L - \bar{\mathbb{A}}^{M}_L)^2}.
\]
The numerical results are illustrated in Figure \ref{fig:Emp_a12_M20}. The calculations for 2D case, 
 depicted in the left panel can be compare with the similar results in \cite{KKO:17} obtained for the case
 of overlapping coefficients sampling. Figure \ref{fig:Emp_a12_M20}, right, confirms the asymptotic 
 estimate in (\ref{eqn:var}) for 3D case.

 \section{Conclusions}\label{sec:concl}

We present the numerical techniques for discretization and fast solution of the 2D and 3D elliptic equations 
with strongly varying piecewise constant coefficients arising  
in numerical analysis of stochastic homogenization problems for multi-scale composite materials.
% and apply this approach in 
% numerical analysis of stochastic homogenization for multiscale composite materials. 
We use random checkerboard coefficient configurations with 
the large size of the RVE, $L$. For a fixed $L$, our method allows to 
avoid the generation of the new FEM space at each stochastic realization.
For every realization, fast assembling of the FEM 
stiffness matrix is performed by agglomerating the Kronecker tensor 
products of 1D FEM discretization matrices.

The spectrally close preconditioner is constructed by using the low Kronecker 
rank approximation to the pseudo-inverse of discrete 2D and  3D periodic Laplacian.
The resulting large linear system of  equations is solved by the preconditioned  
CG iteration  with  the convergence rate that is independent of $L$ and the grid size,
as well as of the variation in jumping coefficients.
The numerical tests illustrate the performance of the Matlab implementation in both 
2D and 3D cases. 

The proposed elliptic problem solver 
 can be applied in the numerical analysis of 3D  stochastic homogenization problems for ergodic processes
 with variable  contrast in random coefficients, 
 for solving numerically the 3D stochastic elliptic PDEs in random heterogeneous materials, 
 for  solution of quasi-periodic (multi-scale) geometric homogenization problems, in  
 the computer simulation of dynamical many body interaction processes and multi-particle electrostatics,
 as well as for numerical analysis of optimal control problems in random media.

\vspace{3mm}

{\bf Acknowledgements} \\

The authors are thankful to Prof. Felix Otto for useful discussions and motivation to develop
an efficient solution scheme for the 3D elliptic PDEs with random coefficients in relation
to numerical simulations for stochastic homogenization problems.

\vspace{3mm}

\begin{footnotesize}

\end{footnotesize}

\end{document}